\documentclass[12pt,leqno,fleqn]{amsart}  
\usepackage{amsmath,amstext,amsthm,amssymb,amsxtra}
\usepackage{txfonts} % pxfonts txfonts 
\usepackage[T1]{fontenc}
\usepackage{lmodern}

\usepackage{euler}   % better than the option below
\usepackage{amsfonts}
\usepackage{latexsym}

\usepackage{tikz}

\newcommand{\R}{\mathbb{R}}

\newcommand{\N}{\mathbb{N}}
\newcommand{\Z}{\mathbb{Z}}

\newcommand{\W}{\mathcal{W}}

\def\diam{\qopname\relax o{diam}}

\def\dist{\qopname\relax o{dist}}
\def\supp{\qopname\relax o{supp}}
\def\min{\qopname\relax o{min}}

\numberwithin{equation}{section}

\usepackage{mathtools}
\mathtoolsset{showonlyrefs,showmanualtags} 
\def\kint_#1{\mathchoice%
          {\mathop{\kern 0.2em\vrule width 0.6em height 0.69678ex depth -0.58065ex
                  \kern -0.8em \intop}\nolimits_{\kern -0.4em#1}}%
          {\mathop{\kern 0.1em\vrule width 0.5em height 0.69678ex depth -0.60387ex
                  \kern -0.6em \intop}\nolimits_{#1}}%
          {\mathop{\kern 0.1em\vrule width 0.5em height 0.69678ex depth -0.60387ex
                  \kern -0.6em \intop}\nolimits_{#1}}%
          {\mathop{\kern 0.1em\vrule width 0.5em height 0.69678ex depth -0.60387ex
                  \kern -0.6em \intop}\nolimits_{#1}}}

\usepackage{hyperref} %,hypertexnames=false,colorlinks,[pagebackref]
\hypersetup{
    colorlinks=true,       % false: boxed links; true: colored links
    linkcolor=blue,          % color of internal links
    citecolor=magenta,        % color of links to bibliography
    filecolor=magenta,      % color of file links
    urlcolor=cyan           % color of external links
}

\allowdisplaybreaks
\swapnumbers

\theoremstyle{plain} % definition 
\newtheorem{lemma}[equation]{Lemma} 
\newtheorem{proposition}[equation]{Proposition} 
\newtheorem{theorem}[equation]{Theorem} 
\newtheorem{corollary}[equation]{Corollary}

\theoremstyle{definition}
\newtheorem{definition}[equation]{Definition} 

\theoremstyle{remark}
\newtheorem{remark}[equation]{Remark}

\numberwithin{equation}{section}

\title[Hardy inequalities]{Hardy inequalities in Triebel--Lizorkin spaces} 


\author{Lizaveta Ihnatsyeva}

\author{Antti V. V\"ah\"akangas}
\thanks{L.I. is supported by the Academy of Finland, grant 252293}
\begin{document}

\begin{abstract}
We prove an inequality of Hardy type for functions in Triebel--Lizorkin spaces. The distance involved is being measured to a given Ahlfors $d$-regular set in $\R^n$, with $n-1<d<n$.
As an application of the Hardy inequality, we consider boundedness of
pointwise multiplication 
operators, and extension problems.
\end{abstract}	
	
\maketitle

\setcounter{tocdepth}{1}
\tableofcontents

%\begin{abstract} 
%\end{abstract}

%%%%%%%%%%%%%%%%%%%%%%%%%%%%%% SECTION  SECTION SECTION
%%%%%%%%%%%%%%%%%%%%%%%%%%%%%% SECTION  SECTION SECTION 

\section{Introduction}\label{introduction}

In the paper we study inequalities of Hardy type 
\begin{equation}\label{hardy_ineq}
\bigg(\int_{\R^n} \frac{\lvert f(x)\rvert^p}{\dist(x,S)^{sp}}\,dx\bigg)^{1/p}
\lesssim \lVert f\rVert_{F^s_{pq}(\R^n)}
\end{equation}
for functions in Triebel--Lizorkin spaces $F^{s}_{pq}(\R^n)$
with vanishing trace on $S$.
The set $S\subset\R^n$, $n\ge 2$, under consideration is supposed to be an (Ahlfors) $d$-regular set with $n-1<d<n$. We also consider applications  to extension problems.

Often Hardy inequalities are formulated in the form
\begin{equation}\label{hardy_ineqDomain}
\bigg(\int_{\Omega} \frac{\lvert f(x)\rvert^p}{\dist(x,\partial\Omega)^{sp}}\,dx\bigg)^{1/p}
\lesssim \lVert f\rVert_{F^s_{pq}(\Omega)},
\end{equation}
where $\Omega$ is an appropriate domain in $\R^n$.
%, and
%$f$ belongs to the closure of $C^\infty_0(\Omega)$ in $F^{s}_{pq}(\Omega)$.
In case of a $C^\infty$ bounded domain, inequality \eqref{hardy_ineqDomain} is known to hold 
for all
$f\in F^{s}_{pq}(\Omega)$ with  $1\le p,q<\infty$ and $0<s<1/p$, see  \cite[p. 58]{T3} and \cite{T99}.
When $s\ge 1/p$, it is natural to require that $f$ can be extended by zero, \cite[Proposition 5.7]{T3}.
We postpone the details on the connection
between \eqref{hardy_ineq} and \eqref{hardy_ineqDomain} to  \S \ref{application}.

Inequality \eqref{hardy_ineqDomain} can be considered as one of the fractional counterparts of 
the classical Hardy inequality (recall that in case of a
$C^\infty$ bounded domain, the Sobolev space $W^{k,p}(\Omega)$ coincides 
with the Triebel--Lizorkin space $F^{k}_{p2}(\Omega)$, $1<p<\infty$ and $k\in\N$). 
While classical Hardy inequalities have been studied under rather general geometric assumptions on the domain (more precisely on its complement, we refer to
\cite{Lewis1988,Wannebo1990,KinnunenMartio}), the fractional analogs were mostly considered on smooth and Lipschitz domains, see \cite{T99,Dyda}.
In more irregular 
geometric settings,
fractional Hardy type inequalities
 have been considered in \cite{E-HSV} and \cite[Proposition 16.5]{T3}. 
Also, Hardy type inequalities for certain Triebel--Lizorkin spaces  
%$\widetilde{F}^{s}_{pq}(\Omega)$
on `E-thick' domains $\Omega$ can be obtained
via so called refined localisation spaces \cite[Theorem 2.18, Proposition 3.10]{T4}.

%Other versions of Hardy inequality of fractional order are stated in terms of Besov norms of a function, see e.g. \cite{T99}.

Let us turn to a discussion of the objects of this paper.
Throughout, we focus on the case of smoothness $s>(n-d)/p$ and $S$ being
a  $d$-set in $\R^n$, $n-1<d<n$.

Our motivation to consider inequality \eqref{hardy_ineq} 
arises from the work  \cite{ihnatsyeva}, where we obtained a 
description of the traces of Besov--Triebel--Lizorkin spaces on  $d$-sets.
This description is given
in terms of a local polynomial approximations \cite{Brudnyi74} and, 
as it turns out, the related tools and porosity based arguments
easily adapt 
to the study of Hardy inequalities, and their consequences. 
Let us remark that we widely use the fact that $d$-sets in $\R^n$ with $d<n$ are porous.
%In \cite{ihnatsyeva} we exploited this fact when proving restriction
%and extension results and, as it turns out, the proof of the restriction theorem in \cite{ihnatsyeva}
%is essentially the same as our proof of the Hardy inequality
%for $d$-sets.

The main result, Theorem \ref{hardy_general}, states that Hardy inequality \eqref{hardy_ineq} is valid if the trace of $f$ on $S$
 is zero pointwise
$\mathcal{H}^d$ almost everywhere.
%
%Now, let us turn to a discussion of the objects of this paper.
%Our main result, Theorem \ref{hardy_general}, shows that Hardy inequality \eqref{hardy_ineq} is valid if $S$ is a $d$-set 
%in $\R^n$ with $n-1<d<n$ and $f\in F^s_{pq}(\R^n)$, $s>(n-d)/p$, $1<p<\infty$, $1<q\le\infty$, is such that
%the trace of $f$ on $S$ is zero pointwise $\mathcal{H}^d$ almost everywhere.
%%assume a zero boundary condition on $S$:---if 
Hardy inequalities are related to the question whether the
characteristic function $\chi_\Omega$ of a domain is a 
pointwise multiplier in $F^{s}_{pq}(\R^n)$.
%, and
%to the question if $C^\infty_0(\Omega)$ 
%is dense in $F^{s}_{pq}(\Omega)$. 
Our auxiliary result, Proposition \ref{ext_hardy}, states that inequality,
\[
\lVert f\chi_\Omega\rVert_{F^{s}_{pq}(\R^n)} \lesssim \lVert f\rVert_{F^{s}_{pq}(\R^n)} + 
\bigg(\int_{\Omega} \frac{\lvert f(x)\rvert^p}{\dist(x,\partial\Omega)^{sp}}\,dx\bigg)^{1/p}\,,
\]
holds for a domain $\Omega$ in $\R^n$ whose
boundary is a $d$-set with $n-1<d<n$.  
As a corollary of Theorem \ref{hardy_general} and Proposition \ref{ext_hardy},
we obtain that $\chi_\Omega$ is a pointwise 
multiplier in the subspace $\{f\in F^{s}_{pq}(\R^n)\,:\,\mathrm{Tr}_{\partial\Omega}f=0\}$, where $\mathrm{Tr}_{\partial\Omega}f$ denotes 
the trace of $f$ on $\partial\Omega$,
see Corollary \ref{ext_corollary}. 
%Density results on $d$-sets, for functions with an order of smoothness
%below the critical case $s<(n-d)/p$, can be found in \cite{caetano}. 
For further results 
on the pointwise multipliers $\chi_\Omega$ of irregular domains, we refer to
\cite{FaracoRogers,sickel,T03}. 

%-----
%Also, the problems of extending Sobolev functions (belonging to $W^{k,p}(\Omega)$ with $k\in\mathbf{N}$)  by zero, and 
%gluing such Sobolev functions with matching traces, are addressed in a recent
%paper \cite{BMMM}.
%In particular, the gluing problem is treated in the case of two-sided $(\epsilon,\delta)$-domains $\Omega$
%with the property that $\partial\Omega$ is a $d$-set for some $d\in [n-1,n)$.
%-----
% Notice that if $\partial\Omega$ is a d-regular 
% set with $n-1<d<n$, then $\Omega^c$ is uniformly $p$-fat for all $p>1$ and this is case well trated and so on...
% Moving attention to more irregular subsets $S$ of $\R^n$, 
% we refer to \cite{Lewis1988,Wannebo1990} for Hardy inequalities in case of $s=1$ and $q=2$,
% and to \cite{T3} for results involving $d$-sets.

As an application, we study the  problem of extending a smooth function from its boundary trace. 
More precisely,
for $f\in F^{s}_{pq}(\R^n)$ with $s>(n-d)/p$
and $k=[s]+1$, we define an extension  \begin{equation}
{\rm Ext}_{k,\Omega}f(x):=\begin{cases}
f(x), \quad &\text{if }x\in \Omega;\\
\sum_{Q\in \mathcal{W}_{\partial\Omega}}\varphi_Q(x)\{\mathrm{Pr}_{k-1,a(Q)}\circ \mathrm{Tr}_{{\partial\Omega}}(f)\}(x), \quad &\text{if }x\in\mathbb{R}^n\setminus \Omega.
\end{cases}
\end{equation}
Let us emphasise
that the operator $\mathrm{Pr}_{k-1,a(Q)}$ is a polynomial 
projector on $L^1(a(Q)\cap \partial\Omega)$, where
$a(Q)$ is a cube, which is centered in $\partial\Omega$ and
is close to the given Whitney cube
$Q\in\mathcal{W}_{\partial\Omega}$.
%For notation, see \S \ref{application}.
Observe that, if $f$ has zero boundary values in $\partial\Omega$, then 
${\rm Ext}_{k,\Omega}f=f\chi_\Omega$.

In Theorem \ref{main_extension}, we prove boundedness of the operator $\mathrm{Ext}_{k,\Omega}$ on $F^{s}_{pq}(\R^n)$.
The proof is based on restriction and extension theorems 
for $d$-sets \cite{ihnatsyeva}, as well as the
multiplier result mentioned above.
To our knowledge, Theorem \ref{main_extension} is not previously known in its generality.
However, the operator $\mathrm{Ext}_{1,\Omega}$ has been considered in connection with
extension results for first order Sobolev spaces $W^{1,p}(\R^n)$, \cite{H}.
The operators $\mathrm{Ext}_{k+1,\Omega}$ with $k\in \{1,2,\ldots\}$
share common properties with certain extension operators of Calder\'on 
\cite{calderon},
which extend
any $f\in W^{k,p}_0(\Omega)$ by zero outside of a given
Lipschitz domain $\Omega$.
Our extension approach
is also related to gluing Sobolev functions with matching traces, \cite{BMMM}.
%Extension of Triebel--Lizorkin functions defined on $n$-sets is  considered
%in \cite{shvartsman}.

%Up to what extent this $d$-set assumption can be relaxed raises interesting questions for further investigation.

\smallskip
The outline of the paper is as follows.
In sections \S \ref{notation}, \S \ref{spaces}, \S \ref{d-sets}, 
and \S \ref{porous_sect} we treat the preliminaries: notation,
definitions of function spaces, $d$-sets, and porosity, respectively.
In \S \ref{hardy_special}, we  prove a special 
 Hardy inequality, namely the one for $F^{s}_{pp}(\R^n)$
with $s\in (0,1)$.
We include
this model result
 to demonstrate the somewhat technical proof of our main theorem 
in a transparent manner.  The main theorem is proven in \S \ref{hardy_general_sect}.
In section \S \ref{application}, we consider applications to pointwise
multipliers and extension problems.

\section{Preliminaries}

\subsection{Notation and auxiliary results}\label{notation}

In this paper, unless otherwise specified, 
we tacitly assume that $n\ge 2$ and
that the distances in $\mathbb{R}^n$ are measured in terms of the supremum norm.
%$\lvert x\rvert = \rVert x\rVert_\infty$ for $x\in\mathbb{R}^n$.
At the same time, by $B(x,r)$ we denote the open Euclidean ball, centered
at $x\in\R^n$ with radius $r>0$.

For a measurable set $E$, with a finite and positive measure, we write
\[
\fint_E f(x)\,dx=\frac{1}{\lvert E \rvert} \int_{E}f(x)\,dx\,.
\]
We write $\chi_E$ for the characteristic function of a set $E$.

%Let $\Omega$ be a bounded domain in $\R^n$, $n\ge 2$, and
%let $\mathcal{W}$ be its Whitney decomposition. 
%For the properties of Whitney cubes $Q\in\mathcal{W}$ we refer to E. M. Stein's book,
%\cite{S}. In particular, we need the inequalities
%\begin{equation}\label{dist_est}
%\diam(Q)\le \dist(Q,\partial \Omega)\le 4 \diam(Q)\,,\quad Q\in \mathcal{W}.
%\end{equation}
%We let $Q\in\mathcal{W}$ be a cube
%with center $x_Q$ and side length $\ell(Q)$.  

By $Q=Q(x_Q,r_Q)$ we denote a closed cube in $\mathbb{R}^n$ centered at $x_Q\in \R^n$ with the
side length $\ell(Q)=2r_Q$ and sides parallel to the coordinate axes.
By $tQ$, $t>0$, we mean
a cube centered 
at $x_Q$ with the
side length $t\ell(Q)$.
We denote by
$\mathcal{D}$ the family of closed dyadic cubes in $\R^n$, and $\mathcal{D}_j$ stands for
the family of those dyadic cubes whose
side length is $2^{-j}$, $j\in \Z$.

%The Lebesgue $n$-measure of a  measurable set $E$ is denoted by $\vert E%\vert.$

The support of a function $f:\R^n\to \mathbb{C}$ is denoted by
$\mathrm{supp}\,f$, and it is the closure of the set  $\{x\,:\,f(x)\not=0\}$ in $\R^n$.
%If $\mu$ is a Borel measure in $\R^n$, 
%the support of $\mu$, denoted by $\mathrm{spt}\,\mu$,
%s the smallest closed set $F$ such that $\mu(\R^n\setminus F)=0$.
%is the Gagliardo seminorm of $f$.

The notation $a\lesssim b$ means that an inequality $a\le cb$ holds for some constant $c>0$
the exact value of which is not important. 
The symbols $c$ and $C$ are used for various positive constants; they may change even on the same line.
%We use subscripts to indicate the dependence on parameters, for example,
%a quantity $c_{d}$ depends on a parameter $d$.

Recall also an inequality of Hardy type for sums.

\begin{lemma}\label{LemmaHardyInequality}
Let $0<p<\infty$ and $a_j\geq 0$, $j=0,1,\dots$. Then
%\begin{equation}\label{leindler}
%\sum_{j=0}^\infty 2^{\sigma j}\bigg(\sum_{i=0}^j a_i\bigg)^p\le c\sum_{j=0}^\infty 2^{\sigma j}a_j^p \,\,\,\,\,\,\text{for}\,\,\sigma<0
%\end{equation}
%and
\begin{equation}\label{Hardy2}
\sum_{j=0}^\infty 2^{\sigma j}\Big(\sum_{i=j}^\infty a_i\Big)^p\lesssim \sum_{j=0}^\infty 2^{\sigma j}a_j^p\,\quad\text{for}\,\,\sigma>0.
\end{equation}
\end{lemma}

For $p>1$ this lemma follows from Leindler's result in \cite{Leindler}. 
The case $0<p\le 1$
 follows by applying the inequality
%$$\bigg(\sum_{i=0}^j a_i\bigg)^p \le \sum_{i=0}^j a_i^p,\qquad 
$(\sum_{i=j}^\infty a_i)^p \le \sum_{i=j}^\infty a_i^p$.

\subsection{Function spaces in $\R^n$}\label{spaces}

%Let $L^p(\R^n)$ denote the Lebesgue space of $p$-integrable functions in $\R^n$. We write \[\|f\|_p=\|f\|_{L^p(\R^n)}:=\bigg(\int_{\R^n} |f(x)|^p dx\bigg)^{1/p}.\]
%For a positive integer $k$ and $1\leq p<\infty$, 
%the Sobolev space $W^{k,p}(\R^n)$ consists of all functions $f\in L^p(\R^n)$ 
%having  distributional derivatives $\partial^s f$,
%$|s|\leq k$, in $L^p(\R^n)$. The Sobolev space is equipped with a norm
%$\|f\|_{W^{k,p}(\R^n)} := \sum_{|s|\le k}\|\partial^s f\|_{p}$.
We recall definition of {\em fractional order Sobolev spaces} in $\R^n$.
For $1\le p<\infty$ and $s\in (0,1)$ we let
$W^{s,p}(\R^n)$ denote the collection of  functions $f$ in
$L^p(\R^n)$ with
\begin{equation}\label{frac_def}
\lVert f \rVert_{W^{s,p}(\R^n)}
:=\lVert f\rVert_{L^p(\R^n)} + \bigg(\int_{\R^n} \int_{\R^n}
\frac{|f(x)-f(y)|^p}{|x-y|^{ps+n}}\,dx\,dy\bigg)^{1/p} < \infty.
\end{equation}
There are several equivalent characterizations for the fractional Sobolev spaces and their natural extensions, Triebel--Lizorkin spaces (see e.g. \cite{AH}, \cite{BesovIlinNikolski}, \cite{T89} and \cite{T2} for the general theory).

 In this
paper, we mostly use definitions based on the local polynomial approximation approach.
Let $f\in L^u_{\mathrm{loc}}(\R^n)$, $1\le u \le \infty$, and $k\in\N_0$.
Following \cite{Brudnyi74}, we define the {\em normalized local best approximation} of $f$ on a cube
$Q$ in $\R^n$ by
\begin{equation} 
\mathcal{E}_k(f,Q)_{L^u(\R^n)}:=\inf_{P\in \mathcal{P}_{k-1}}\bigg(\frac{1}{|Q|}\int_{Q}|f(x)-P(x)|^u\;dx\bigg)^{1/u}.
\end{equation}
Here and below  $\mathcal{P}_{k}$, $k\geq 0$, denotes the space of polynomials on $\mathbb{R}^n$ of degree at most $k$. 
Let $Q_1\subset Q_2$ be two cubes in $\R^n$. Then
\begin{equation}\label{eqMonotonyOfLocalApproxSset}
\mathcal{E}_k(f,Q_1)_{L^u(\R^n)}
\le \bigg(\frac{r_{Q_2}}{r_{Q_1}}\bigg)^{n/u}
\mathcal{E}_k(f,Q_2)_{L^u(\R^n)}.
\end{equation}
This property is referred as  {\em the monotonicity
of local approximation}.

%For convenience,
%we denote $\mathcal{P}_{-1}=\{0\}$.

%For $s>0$ and $1\le p,q\le\infty$, the Besov space $B_{pq}^{s}(\R^n)$ can be defined in the following way, see e.g. \cite{Brudnyi74}, \cite{Triebel89}, \cite{Triebel2}.
%Let $k$ be an integer such that $s<k$ and $1\le u\le p$, then $B_{pq}^{s}(\R^n)$ consists of functions $f\in L^p(\R^n)$ such that
%\begin{equation}\label{DefofBesovSpace}
%\int_0^1\bigg(
%\frac{\Vert\mathcal{E}_k(f,Q(\cdot,t))_{L^u(\R^n)}\Vert_{L^p(\R^n)}}{t^{s}}\bigg)^{q}\;\frac{dt}{t}<\infty,\qquad 
%\text{ if }q<\infty,
%\end{equation}
%and $\sup_{0<t\le 1} t^{-s}\Vert\mathcal{E}_k(f,Q(\cdot,t))_{L^u(\R^n)}\Vert_{L^p(\R^n)}<\infty$ if $q=\infty$. The Besov norms
%$$
%\|f\|_{B^s_{pq}(\R^n)}:=\|f\|_{L^p(\R^n)}+\bigg(\int_0^1\bigg(
%\frac{\Vert\mathcal{E}_k(f,Q(\cdot,t))_{L^u(\R^n)}\Vert_{L^p(\R^n)}}{t^{s}}\bigg)^{q}\;\frac{dt}{t}\bigg)^{1/q}
%$$
%(modification
%if $q=\infty$) are equivalent if $1\le u\le p$ and $s<k$.

The following definition of Triebel--Lizorkin spaces 
with positive smoothness
can be found in \cite{T89}.
Let $s>0$, $1\le p<\infty$, $1\le q\le\infty$, and $k$ be an integer such that $s<k$. 
For $f\in L^u_{\mathrm{loc}}(\R^n)$, $1\le u\le \min\{p,q\}$, set for 
all $x\in\R^n$,
\[
g(x):=\bigg(\int_0^1\bigg(\frac{\mathcal{E}_k(f,Q(x,t))_{L^u(\R^n)}}{t^{s}}\bigg)^{q}\;\frac{dt}{t}\bigg)^{1/q} ,\qquad \text{if }q<\infty,
\]
and $g(x):=\sup\{t^{-s}\mathcal{E}_k(f,Q(x,t))_{L^u(\R^n)}:\, 0<t\le 1\}$ if $q=\infty$. 
The function $f$ belongs to
Triebel--Lizorkin space $F^s_{pq}(\R^n)$ if $f$ and $g$ are both in $L^p(\R^n)$. The Triebel--Lizorkin norms
\[
\| f\|_{F^s_{pq}(\R^n)}:=\|f\|_{L^p(\R^n)}+\|g\|_{L^p(\R^n)}
\]
are equivalent if $s<k$ and $1\le u\le \min\{p,q\}$. In particular, if $q\ge p$, then
we can set $u=p$.

Triebel--Lizorkin spaces include fractional Sobolev spaces as a special case:
% $F_{p2}^{k}(\R^n)$ coincides with $W^{k,p}(\R^n)$ for any $k\in\N$ and $1<p<\infty$, see e.g. \cite{AdamsHedberg} for a proof.
if $1<p<\infty$ and $0<s<1$, then $F^{s}_{pp}(\R^n)$ 
coincides with $W^{s,p}(\R^n)$, \cite[Theorem 6.7]{ds}  and
\cite[pp. 6--7]{T2}.

\subsection{$d$-sets and inequalities}\label{d-sets}
Recall definition of an Ahlfors $d$-regular set,
or, briefly,  $d$-set.

\begin{definition}\label{sset_def}
Let $0<d\le n$.
A closed set $S\subset \R^n$ is a {\em $d$-set}  if there is a constant $C>1$ such that
$C^{-1}r^d \le \mathcal{H}^d(Q(w,r)\cap S)\le Cr^d$
for every $w\in S$ and $0<r\le 1$.
\end{definition}

Note that if $S$ is a $d$-set, then
\begin{equation}\label{sset_rem}
c^{-1}r^d\le \mathcal{H}^d(Q(w,r)\cap S)\le cr^d 
\end{equation}
for every $w\in S$ and $0<r\le R$, where $R$ is any fixed positive number and the constant $c\ge 1$ depends on
parameters $R,C,d,n$.

%\begin{remark}
%Sometimes the definition of a $d$-set is formulated in terms of a Borel measure $\mu$ in $\R^n$, satisfying
%the conditions $\mathrm{supp}\,\mu =S$ and
%\eqref{sset} with $\mathcal{H}^d$ replaced by $\mu$. Such a measure $\mu$ can be identified, up to constants $c_1$ and $c_2$, with the measure
%$\mathcal{H}^d|_S$; a short proof can be found in \cite[Theorem 3.4]{Triebelfs}.
%This equivalence of measures also
%implies that \[\mathcal{H}^d(\bar S\setminus \mathrm{int}\,S)=0\] (see \cite[pp. 205--206]{JonssonWallin1984}) and, therefore, we do not lose generality when we consider only closed sets. 
%\end{remark}

\medskip
We write  $L^p(S)$ for the space of $p$-integrable functions on a 
$d$-set $S$ with respect to the natural Hausdorff measure $\mathcal{H}^d|_S$.

The following is a Remez type theorem, \cite{BrudnyiBrudnyi}.

\begin{theorem}\label{qremez}
Let $S\subset \mathbb{R}^n$ be a $d$-set, $n-1 < d\le n$. Suppose that $Q=Q(x_Q,r_Q)$ and $Q'=Q(x_{Q'},r_{Q'})$ 
are cubes  in $\R^n$ such that
$x_{Q'}\in S$, $Q'\subset Q$, and \[0<r_Q\le R\,r_{Q'}\le R^2\] for some $R>1$. Then,
for every polynomial $p$ of degree at most $k$, we have
\[
\bigg(\frac{1}{|Q|}\int_{Q}|p|^r\;dx\bigg)^{1/r}\le C
\bigg(\frac{1}{\mathcal{H}^d(Q'\cap S)}\int_{Q'\cap S}|p|^u\;d\mathcal{H}^d\bigg)^{1/u},
\]
where $1\le u,r\le\infty$ and the constant $C$
depends on $S,\,R,\,n,\,u,\,r,\,k$.
\end{theorem}

We  also need the following energy type estimate. 

\begin{lemma}\label{riesz_est}
Let $S\subset \mathbb{R}^n$ be a $d$-set, $n-1 < d<n$.
Suppose that  $1<r<\infty$ and $0<\omega<1$ satisfy
$\omega r>n-d$. Then,
for every point $x\in \partial\Omega$ and $0<t\le T\in \R$,
\begin{equation}\label{riesz_integral}
\bigg( \int_{Q(x,t)} \bigg\lvert
\int_{Q(y,t)\cap S} \frac{d\mathcal{H}^d (z)}{\lvert y-z\rvert^{n-\omega}}\,\bigg\rvert^{r'}\,dy\bigg)^{1/r'} 
\lesssim t^{d+\omega-n/r}\,,
\end{equation}
where $1/r+1/{r'}=1$.
\end{lemma}

\begin{proof}
Let us begin with two auxiliary estimates.
%valid  (at least) in case of $0<t\le T$.
The first one,
\begin{equation}
\sup_{y\in\R^n}\int_{Q(y,t)\cap S} \frac{d\mathcal{H}^d(z)}{\lvert y-z\rvert^{\sigma}}
\lesssim t^{d-\sigma}\,,\quad 0<\sigma<d\,,
\end{equation}
follows easily from the
layer cake representation
and the definition of a $d$-set, \cite[p. 104]{JonssonWallin1984}. 

The
second one,
\begin{equation}\label{assouad}
\sup_{x\in S}\int_{Q(x,t)} \dist(y,S)^{\mu-n}\,dy\lesssim 
t^{\mu}\,,\quad d<\mu<n\,,
\end{equation}
is closely related to Aikawa dimension of $S$, and its proof 
can be found in \cite[Lemma 2.1]{lehrback}.

Denote the left hand side of \eqref{riesz_integral}
by $LHS$. Choose an auxiliary parameter $\beta>0$ such that
\begin{equation}\label{auxiliary}
0< n-\omega-\beta < d,\qquad \beta r'<n-d.
\end{equation}
The assumptions of the lemma ensure that this can be done (e.g. $\beta=(n-d)/r'-\varepsilon$ with suitable $\varepsilon>0$).

By the estimates above and the choice of $\beta$,
\begin{align*}
LHS&\le 
\bigg( \int_{Q(x,t)}
\dist(y,S)^{-\beta r'}
 \bigg\lvert
\int_{Q(y,t)\cap S} \frac{d\mathcal{H}^d (z)}{\lvert y-z\rvert^{n-\omega-\beta}}\bigg\rvert^{r'}\,dy\bigg)^{1/r'}\\
&\lesssim  t^{d-n+\omega+\beta + (n-\beta r')/r'}\,.
\end{align*}
A simplification of exponents finishes the proof.
\end{proof}

\subsection{Porous sets and Whitney decomposition} \label{porous_sect}

In this paper we widely use the fact that $d$-sets with $d<n$ are porous. 
The treatment here follows parts of \cite{ihnatsyeva}.

\begin{definition}\label{porous}
A  set $S\subset\R^n$ is 
{\em porous} (or {\em $\kappa$-porous}) if for some $\kappa\ge 1$
the following statement is true:
For every cube $Q(x,r)$ with $x\in\R^n$ and $0<r\le 1$ there is $y\in Q(x,r)$ such that
$Q(y,r/\kappa)\cap S=\emptyset$.
\end{definition}

\begin{remark}\label{porous_obs}
The observation that $d$-sets with $d<n$ are porous was already done in \cite{JonssonII}. See also Proposition 9.18 in \cite{T3} which gives this fact as a special case.
%Let us also mention that a set $S$ in $\R^n$
%is porous if, and only if, its Assouad dimension strictly less 
%than $n$, \cite{Luukkainen}. 
%An interesting recent result is that the Assouad and Aikawa 
%dimensions of sets $S\subset \R^n$ 
%coincide, \cite{lehrbackII}.
\end{remark}

Here we recall a 
reverse H\"older type inequality involving porous sets, Theorem \ref{reverse}.
For a set $S$ in $\R^n$ and a positive constant $\gamma>0$ we denote
\begin{equation}\label{c_definition}
\mathcal{C}_{S,\gamma} = \{Q\in\mathcal{D}\,:\,\gamma^{-1}\mathrm{dist}(x_Q,S)\le \ell(Q)\le 1\}.
\end{equation}
This is the family of dyadic cubes that are relatively close to the set.
Such families arise naturally while treating the `extension by zero' -problem, 
we refer to later Proposition \ref{ext_hardy}.

\begin{theorem}\label{reverse}
Suppose that $S\subset\R^n$ is porous.
Let $p,q\in (1,\infty)$ and $\{a_Q\}_{Q\in\mathcal{C}_{S,\gamma}}$ be a sequence of non-negative scalars. 
Then
\begin{equation}\label{rev}
\bigg\|\sum_{Q\in\mathcal{C}_{S,\gamma}} \chi_Q a_Q\bigg\|_p\le c\bigg\|\bigg(\sum_{Q\in\mathcal{C}_{S,\gamma}}
(\chi_{Q} a_Q)^q\bigg)^{1/q}\bigg\|_p.
\end{equation}
Here the constant $c$ depends on $n$, $p$, $\gamma$ and the set $S$.
\end{theorem}

The proof of this theorem
is a consequence of maximal-function techniques, we refer to \cite{ihnatsyeva}.
% that can be found in \cite{Bojarski} and \cite{iwaniec}. 

\medskip
%We connect porosity to Whitney cubes.
Supposing that $S$ is a non-trivial closed set in $\R^n$, its complement has a Whitney decomposition, see e.g. \cite{S}.
That is, there is a family  $\mathcal{W}_{S}$ of  dyadic cubes
whose interiors are pairwise disjoint and 
$\mathbb{R}^n\setminus S=\bigcup_{Q\in \mathcal{W}_{S}}Q$. Furthermore,
if $Q\in \mathcal{W}_{S}$, then
\begin{equation}\label{dist_est}
\diam(Q)\le {\rm dist}(Q,S)\le 4\diam(Q).
\end{equation}
It is easy to check that the standard construction, usually given in terms of the Euclidean metric, admits this modification where we use the uniform metric instead.

\medskip

Let us write $\ell(Q)\overset{\kappa}{\sim} 2^{-i}$ if
$Q$ is a cube in $\R^n$ and $2^{-i}/5\kappa\le \diam(Q)\le 2^{-i}$.

\begin{lemma}\label{cop}
Suppose that $S\subset \R^n$ is $\kappa$-porous. Let
 $x\in S$ and $i\in\N$. Then there is $Q\in\mathcal{W}_{S}$ such that
$\ell(Q)\overset{\kappa}{\sim} 2^{-(i+1)}$ and 
$Q\subset Q(x,2^{-i})$. 
\end{lemma}

\begin{proof}
By $\kappa$-porosity of $S$, 
there is $y\in Q(x,2^{-i-1})$ such that
$Q(y,2^{-i-1}/\kappa)\subset \R^n\setminus S$.
Let $Q\in\mathcal{W}_{S}$ be a cube containing the point $y$. Then
\[
2^{-i-1}/{\kappa}-\diam(Q)\le \mathrm{dist}(y,S)-\diam(Q)\le \mathrm{dist}(Q,S).
\]
On the other hand,
\[
\mathrm{dist}(Q,S)\le \mathrm{dist}(y,S)\le \|x-y\|_\infty\le 2^{-i-1}.
\]
By \eqref{dist_est}, 
$2^{-i-1}/5\kappa\le \diam(Q)\le 2^{-i-1}$, that is,
$\ell(Q)\overset{\kappa}{\sim} 2^{-(i+1)}$. 

It is also easy to see that $Q\subset Q(x,2^{-i})$.
\end{proof}
%and 
%\begin{equation}\label{i)}
%a(Q)\subset \mathrm{int}(10Q),\qquad  Q\in\mathcal{W}_S.
%\end{equation}
%Indeed, \eqref{ii)}   follows from \eqref{sset_rem}, 
%and the constant $c>0$ depends on  $\Delta,S$ and $n$.
%Furthermore, if 
%$y\in a(Q)$, then \begin{align*}
%\|y-x_Q\|_\infty&\le \|y-a_Q\|_\infty+\|a_Q-x_Q\|_\infty\\&< r_Q+{\rm dist}(x_Q,S)\le 2r_Q+{\rm dist}(Q,S),
%\end{align*}
%and, since
%$\mathrm{dist}(Q,S)\le 4 \diam(Q)=8 r_Q$,  \eqref{i)} is true.

\section{Hardy inequalities}

In \S \ref{hardy_general_sect}
we prove our main result, Theorem \ref{hardy_general}.
Surprisingly, the proof of a restriction
theorem \cite[Theorem 4.8]{ihnatsyeva} can be modified
to yield a proof
of this Hardy type inequality. This proof is still  technical, and for this reason
in \S \ref{hardy_special} we consider an
illustrative special case whose proof is analogous
but involves less technicalities.
This simpler proof is based on \cite{E-HSV}, the main difference
being that the
energy type estimate in Lemma \ref{riesz_est} replaces
certain capacitary considerations. 
In order to make this
paper self contained, we repeat some of the arguments
in the aforementioned papers.

We say that a locally integrable function $f$ is strictly defined at $x$ if the limit
\[
\bar{f}(x)=\lim_{r\to 0+}  \fint_{Q(x,r)}f(y)\,dy
\]
exists. Observe that $f$ is strictly defined at its
Lebesgue points and, by the Lebesgue differentiation theorem, $f=\bar{f}$ a.e in $\R^n$.

Let $S$ be a $d$-set in $\R^n$, $n-1<d<n$.
At those points $x\in S$, in which $\bar{f}(x)$ exists, we define the  trace of the function $f$ on $S$ by 
\[\mathrm{Tr}_{S} f(x):=\bar{f}(x)\,.\]
Assume that
$f$ belongs to the Triebel--Lizorkin space $F_{pq}^{s}(\R^n)$
with $s>(n-d)/p$, $1< p<\infty$, $1\le q\le \infty$.
Then the trace of $f$
is defined $\mathcal{H}^d$-almost everywhere on $S$.
In fact, 
under these assumptions, the exceptional set for the Lebesgue points of $\bar f$ in $\R^n$ has zero $d$-dimensional Hausdorff measure.
In order to see this, apply the trivial embeddings 
$F^{s}_{pq}(\R^n)\subset F^{s-\epsilon}_{p2}(\R^n)$ with
$0<\epsilon<s-(n-d)/p$ and use 
\cite[Theorem 1.2.4, Theorem 6.2.1]{AH}.

\subsection{The case of $0<s<1$ and $q\le p$}\label{hardy_special}
%We will prove the following fractional Hardy inequality.
As a corollary of Theorem \ref{fractional_hardy_fat}, 
we obtain Hardy type inequalities for Triebel--Lizorkin
spaces $F^{s}_{pq}(\R^n)$, $0<s<1$, $1\le q\le p$. Indeed, this follows
from boundedness of relations
$F^{s}_{pq}(\R^n)\subset F^{s}_{pp}(\R^n)\simeq W^{s,p}(\R^n)$, \cite[pp. 6--7]{T2}.
Recall also \eqref{frac_def}.

%Therefore we can modify the proof of a certain Hardy-type inequality
%to fit our purposes, \cite{E-HSV}.

\begin{theorem}\label{fractional_hardy_fat}
Suppose that $S$ is a $d$-set in $\R^n$ and $(n-d)/p<s<1$, where $n-1<d<n$ and 
$1<p<\infty$.
Let $f\in W^{s,p}(\R^n)$
be such
that $\mathrm{Tr}_{S} f = 0$ pointwise
$\mathcal{H}^d$ almost everywhere.
Then $f$ satisfies the Hardy type inequality,
\[
\bigg(\int_{\R^n} \frac{\lvert f(x)\rvert^p}{\dist(x,S)^{sp}}dx\bigg)^{1/p}
\lesssim \lVert  f \rVert_{W^{s,p}(\R^n)}\,.
\]
\end{theorem}
%Let $0<\sigma <n$.
%The Riesz potential of a 
% a Borel measure $\mu$ on $\R^n$ is given by
%\[
% I_\sigma \mu(x)=\int_{\R^n} \frac{1}{|x-y|^{n-\sigma}}\,d\mu(y)\,.
%\]

%Therefore we can modify the proof of a certain Hardy-type inequality
%to fit our purposes, \cite{E-HSV}.
The proof is based upon \cite{E-HSV}.

We need some preparations. Let us first
recall fractional  inequalities
for functions defined on cubes, 
 Lemma~\ref{integral_representation} and
Lemma~\ref{hardy_cube}.
Both inequalities are invariant under
scaling and translation.

The following lemma follows from the proof of
\cite[Theorem 4.10]{H-SV}.

\begin{lemma}\label{integral_representation}
Let $\sigma\in (0,1)$ and  $r\in (1,\infty)$. 
Let $u\in L^r_{\textup{loc}}(Q)$ for a cube  $Q$ in $\R^n$, $n\ge 2$. Then
\begin{equation}
\vert u(x)-u_{B_Q}\vert \le c_{n,\sigma,r}
\int_{Q}\frac{g_{Q}(y)}{\vert x-y\vert ^{n-\sigma}}\,dy
\end{equation}
if $x\in Q$ is a Lebesgue point of $u$.
Here  $B_Q:=B(x_Q,\ell(Q)/c_n)$, $c_n>2$, and the function
$g_Q$ is defined by
 \begin{equation*}
g_{Q}(y)=\biggl(\int_{Q}\frac{\vert u(y)-u(z)\vert ^r}{\vert y-z\vert ^{n+\sigma r}}\,dz\biggr)^{1/r}\,.
\end{equation*}
\end{lemma}

The following 
fractional Sobolev--Poincar\'e inequality is a consequence of
\cite[Remark 4.14]{H-SV}.
%, when $Q=[-1/2,1/2]^n.$
%A change of variables gives the general case.

\begin{lemma}\label{hardy_cube}
Let $Q$ be a cube in $\R^n$, $n\ge 2$.
Suppose that $p,r\in [1,\infty)$, and $s \in (0,1)$ satisfy
%\begin{itemize}
%\item[{\em a)}] $1<p<n/s $ and $q= np/(n-s  p)$;
$0\le 1/r-1/p<s /n$.
Then, for every $u\in L^r(Q)$,
\[
\int_{Q} |u(x)-u_Q|^p\,dx \le c|Q|^{1+ps  /n-p/r} \bigg(\int_{Q}\int_{Q} 
\frac{|u(x)-u(y)|^r}{|x-y|^{n+sr}}\,dy\,dx\bigg)^{p/r}.
\]
Here the constant $c>0$ is independent of $Q$ and $u$.
\end{lemma}

%Let us 
%recall that
%a maximal function of a locally integrable function
%$f:\R^{m}\to [-\infty,\infty]$ is
%\[
%\mathcal{M} f(x) = \sup_{r>0} \fint_{B(x,r)} |f(y)|\,dy.
%\]
%If $Q$ is a cube in $\R^m$ and $x\in Q$, then
%\begin{equation}\label{cube_est}
%  \fint_{Q} |f(y)|\,dy\le c_m\mathcal{M} f(x).
%\end{equation}
The last lemma we need is the following.

\begin{lemma}\label{carleson_lemma}
Assume that $S$ is a non-trivial closed set in $\R^n$, $n\ge 2$,
and let $\mathcal{W}_S$ be a
Whitney decomposition of $\R^n\setminus S$.
Suppose that $1<s<\infty$ and $\kappa\ge 1$. Then
\begin{equation}\label{beta_ineq}
\begin{split}
%\begin{align}\label{beta_ineq}
&\sum_{Q\in\mathcal{W}_S} |\kappa Q|^{2}\bigg(\fint_{\kappa Q}\fint_{\kappa Q} |g(x,y)|\,dx\,dy\bigg)^s
\\&\qquad\qquad\qquad\le c_{n,s} \kappa^{n}\int_{\R^n}\int_{\R^n}  |g(x,y)|^s\,dx\,dy
\end{split}
\end{equation}
%\end{align}
for every $g\in L^s(\R^n\times \R^n)$.
\end{lemma}

\begin{proof}
Throughout this proof, we denote $m=2n$.
Recall 
a (non-centered) maximal function of a locally integrable function
$f:\R^{m}\to [-\infty,\infty]$,
\[
\mathcal{M} f(x) = \sup_{x\in Q} \fint_{Q} |f(y)|\,dy\,.
\]
The supremum is taken over all cubes in
$\mathbb{R}^m$ containing $x\in \mathbb{R}^m$.

Let us rewrite the left hand side of inequality \eqref{beta_ineq} as
\begin{align*}
LHS=\kappa^{n}&\sum_{Q\in\mathcal{W}_S} \int_{\R^n}\int_{\R^n} 
\chi_{\kappa Q}(z)\chi_{Q}(w)
 \bigg(\fint_{\kappa Q} \fint_{\kappa Q}|g(x,y)|\,dx\,dy\bigg)^s
\,dz\,dw.
\end{align*}
By definition of maximal function,
\begin{align*}
\kappa^{-n}LHS&\lesssim \sum_{Q\in\mathcal{W}_S}\int_{\R^n}\int_{\R^n} \chi_{\kappa Q}(z)\chi_Q(w) \big\{\mathcal{M} g(z,w)\big\}^s\,dz\,dw\\
&\lesssim \int_{\R^n}\int_{\R^n} \big\{\mathcal{M} g(z,w)\big\}^s\,dz\,dw\,.
\end{align*}
The boundedness of maximal operator
on $L^s(R^m)$ yields
the right hand side of inequality \eqref{beta_ineq}.
\end{proof}

\begin{proof}[Proof of Theorem \ref{fractional_hardy_fat}]
Fix a number $r\in (1,p)$ such that $\omega=n(1/p-1/r)+s \in (0,s )$
and
 \begin{equation}\label{mu2_def}
n-d<\omega r <sp,\qquad 0<1/r-1/p<\omega/n\,.
 \end{equation}
 For each  $Q\in\mathcal{W}_S$ with $\ell(Q)\le 1$, we write
$Q^*=12 Q$. We let
$B_{Q^*}=B(x_Q,\ell(Q^*)/c_n)$ be 
given by Lemma \ref{integral_representation}.
The proof proceeds with an application of inequality,
\begin{equation}\label{upper_bound}
\begin{split}
\int_{Q} &\lvert f(x)-f_{B_{Q^*}}\rvert^p\,dx +
\lvert Q\rvert \lvert f_{B_{Q^*}}\rvert^p\\
&\lesssim
|Q^*|^{1+p\omega /n-p/r} \bigg(\int_{Q^*}\int_{Q^*} 
\frac{|f(x)-f(y)|^r}{|x-y|^{n+\omega r}}\,dy\,dx\bigg)^{p/r}.
\end{split}
\end{equation}
Let us postpone the proof
of this inequality for the time being, and finish with the main line of the argument first.
By property \eqref{dist_est} of Whitney cubes and inequality \eqref{upper_bound},
\begin{align*}
I:&=\int_{\R^n}
\frac{\vert f(x)\rvert^p}
{\dist (x,S)^{sp}} \,dx\\
&\le
\sum _{Q\in \W_S}
\diam (Q)^{-sp}
\int_{Q}\vert f(x)\vert^p\,dx\\
%&\lesssim \lVert f\rVert_{L^p(\R^n)}^p + 
%\sum _{Q\in \W}
%|Q^*|^{1+ p(\omega /n-1/r-s/n)} \bigg(\int_{Q^*}\int_{Q^*} 
%\frac{|u(x)-u(y)|^r}{|x-y|^{n+\omega r}}\,dy\,dx\bigg)^{p/r}\\
%&\lesssim
%\sum _{Q\in \W}
%|Q|^{1+ q(-s/n +s  /n-2/r+1/p)} \bigg(\int_{Q}\int_{Q} 
%\frac{|u(x)-u(y)|^r}{|x-y|^{nr/p+s  r}}\,dy\,dx\bigg)^{q/r}\\
&\lesssim \lVert f\rVert_{L^p(\R^n)}^p + 
\sum _{\substack{Q\in \W_S\\\ell(Q)\le 1}}
|Q^*|^{1+ p(\omega /n+1/r-s/n)} \bigg(\fint_{Q^*}\fint_{Q^*} 
\frac{|f(x)-f(y)|^r}{|x-y|^{n+\omega r}}\,dy\,dx\bigg)^{p/r}\,.
\end{align*}
The definition of $\omega$ implies that
\begin{equation*}
1+ p(\omega /n+1/r-s/n)= 2\,.
\end{equation*}
Thus, Lemma \ref{carleson_lemma} with $s=p/r>1$ yields
\begin{align*}
%&\lesssim
%\sum _{Q\in \W}
%|Q|^{1+ q(-s/n +\omega /n+1/r)} \bigg(\fint_{Q}\fint_{Q} 
%\frac{|u(x)-u(y)|^r}{|x-y|^{n+\omega r}}\,dy\,dx\bigg)^{q/r}\\
 %&\lesssim 
%\lVert f\rVert_{L^p(\R^n)}^p + 
%\sum _{Q\in \W_S}
%|Q^*|^{2} \bigg(\fint_{Q^*}\fint_{Q^*} 
%\frac{|f(x)-f(y)|^r}{|x-y|^{n+\omega r}}\,dy\,dx\bigg)^{p/r}\\
I&\lesssim \lVert f\rVert_{L^p(\R^n)}^p+
\int_{\R^n}\int_{\R^n} 
\frac{\vert f(x)-f(y)\vert^p}{\vert x-y\vert ^{ps  +n}}
\,dx\,dy \lesssim \lVert f\rVert_{W^{s,p}(\R^n)}^p\,.
\end{align*}
Here we also used the identity $p(n+\omega r)/r = ps+n$.

It remains to verify inequality \eqref{upper_bound}.
Using the facts that $B_{Q^*}\subset Q^*$ and the measures are comparable,
this bound for the integral term is a consequence of Lemma \ref{hardy_cube}.
Hence, it is enough to estimate the remaining term 
$\lvert Q\rvert \lvert f_{B_{Q^*}}\rvert^p$.

Let us fix  $y_Q\in S$ such that
$\dist(y_Q,Q)=\dist(Q,S)$. 
%\[
%B(y_Q,\ell(Q))\subset Q^*\,. %\quad r_Q:=r_0 \ell(Q)/(r_0+1+\diam(\Omega))\,.
%\]
Denote 
\[
E_Q=Q(y_Q,\ell(Q))\cap S\,,
\] and 
define
\[
g_{Q^*}(y)=\biggl(\int_{Q^*}\frac{\vert f(y)-f(z)\vert ^r}{\vert y-z\vert ^{n+\omega r}}\,dz\biggr)^{1/r}\,.
\]
Recall that
$\mathcal{H}^d$ almost every point $x\in S$ is both
a Lebesgue point of $\bar f$, and satisfies $\bar f(x)=0$.
Thus, by 
the fact that $ E_Q\subset Q^*$ and
Lemma \ref{integral_representation} applied
with $\bar f$,
\begin{align*}
\lvert f_{B_{Q^*}}\rvert\cdot \mathcal{H}^d(E_Q) &=\int_{E_Q} |\bar f(x)-\bar f_{B_{Q^*}}\rvert\,d\mathcal{H}^d(x)\\
%&\lesssim \int_{E_Q} \int_{Q^*}\frac{g_{Q^*}(y)}{|x-y|^{n-\omega}}\,dy\,d\mu(x)\\
&\lesssim \int_{Q^*} 
\int_{E_Q} \frac{d\mathcal{H}^d(x)}{\lvert x-y\rvert^{n-\omega}}\,
g_{Q^*}(y) \,dy\,.
\end{align*}
Observe that
$Q^*\subset Q(y_Q,t)$ and
$E_Q\subset Q(y,t)$
for every $y\in Q^*$, where $t=2\diam(Q^*)$.
Thus, by H\"older's inequality, followed by
Lemma \ref{riesz_est} with $x=y_Q$, we obtain
\begin{align*}
\lvert f_{B_{Q^*}}\rvert\mathcal{H}^d(E_Q)
%&\lesssim \int_{E_Q} \int_{Q^*}\frac{g_{Q^*}(y)}{|x-y|^{n-\omega}}\,dy\,d\mu(x)\\
&\lesssim \lVert g_{Q^*}\chi_{Q^*}\rVert_{r}
\cdot t^{d+\omega-n/r}\,.
\end{align*}
Since $t^d\lesssim \mathcal{H}^d(E_Q)$, the upper bound in \eqref{upper_bound} for 
$\lvert Q\rvert \lvert f_{B_{Q^*}}\rvert^p$ follows.
\end{proof}

\subsection{The general case}\label{hardy_general_sect}
The following theorem is our main result.

\begin{theorem}\label{hardy_general}
Let $S$ be a $d$-set in $\R^n$, $n-1<d<n$.
Suppose that  $1<p<\infty$, $1\le q\le \infty$, and
$s>(n-d)/p$. 
Let $f\in F^{s}_{pq}(\R^n)$
be such
that $\mathrm{Tr}_{S} f = 0$ pointwise
$\mathcal{H}^d$ almost everywhere.
Then $f$ satisfies the Hardy type inequality,
\[
\bigg(\int_{\R^n} \frac{\lvert f(x)\rvert^p}{\dist(x,S)^{sp}}dx\bigg)^{1/p}
\le c\lVert  f \rVert_{F^s_{pq}(\R^n)}\,.
\]
The constant $c>0$ depends on $n$, $d$, $p$, $s$, and $S$.
\end{theorem}

We will need several preparations.

Let $Q_x^j$ denote a cube $Q(x,2^{-j})$ with $x\in S$ and $j\in\Z$. Let 
$P_{Q_x^j}$
 be a
projection from $L^p(Q_x^j)$ to $\mathcal{P}_{k-1}$ 
with $k=[s]+1$
such that $\mathcal{E}_k(f,Q_x^j)_{L^p(\R^n)}$
is equivalent to
\[\bigg(\fint_{Q_x^j}\lvert f-P_{Q_x^j}f\rvert^p\,dy\bigg)^{1/p}\,.\]
For the construction of these projections, we refer to
\cite[Proposition 3.4]{Shvartsman} and \cite{DeVoreSharpley}.
We use the following properties of these polynomial projections:
\begin{itemize}
%\item[i)]$P_Q(\lambda)=\lambda\,\,\text{for any} \,\,\lambda\in \mathbb{R};$ 
%\item[ii)]$|P_Q f(y)|\leq c|Q|^{-1}\int_Q |f|$ if $y\in Q$;
\item[i)]If $Q'\subset Q$ are cubes as above and $\lvert Q'\rvert \geq c\lvert Q\rvert$, then for every $z\in{Q'}$
\[\lvert P_Qf(z)-P_{Q'}f(z)\rvert \le C(c,n,k) \fint_{Q}\lvert f-P_Q f\rvert\,dy;\]
\item[ii)]$\lim_{j\to \infty}P_{Q_z^j}f(z)=f(z)$ at every Lebesgue point $z\in S$  of $f$.
\end{itemize}

Let $\mathcal{W}_S$ denote the Whitney
decomposition of $\R^n\setminus S$.

To every  $Q=Q(x_Q,r_Q)$ in $\mathcal{W}_S$, assign the cube
$
a(Q):=Q(a_Q,r_Q/2),
$
where $a_Q\in S$ is such that
$\|x_Q-a_Q\|_\infty =\mathrm{dist}(x_Q,S)$. 
%Let $\Delta\ge 1$ be a parameter which will be chosen later. 
Then 
\begin{equation}\label{ii)}
\mathcal{H}^d(a(Q)\cap S)\geq 
c\, r_Q^d,\qquad \text{if }\diam(Q)\le 1\,.
\end{equation}
This follows from Definition \ref{sset_def}. 
The constant $c>0$ depends on  $S$.

\begin{lemma}\label{EqLocalApInLpSViaLocalApInLp}
Let $f\in L^p(\R^n)$, $1< p<\infty$, be a function for which
$\mathcal{H}^d$-almost every point in a given $d$-set $S$ is a Lebesgue point,
$n-1<d<n$.
Then, for every $i\in \N$, we have
\begin{equation}\label{EqLocalApInLpSViaLocalApInLp_est}
\begin{split}
&\bigg\{\sum_{Q\in\mathcal{W}_{S}\cap \mathcal{D}_{i+5}} \fint_{32a(Q)\cap S}
\lvert f- P_{32a(Q)} f\rvert^p\,d\mathcal{H}^d\bigg\}^{1/p}\\
&\lesssim 2^{id/p}
\sum_{j=i}^\infty\Vert\mathcal{E}_k(f,Q(\cdot,2^{-j}))_{L^p(\R^n)}\Vert_{L^p(S)}. 
\end{split}
\end{equation}
\end{lemma}

\begin{proof}
%Properties i)--i) follow from the construction of the projections. The statement ii) can then be derived from i) and ii)
%as follows
%\begin{equation*}
%|P_{Q(x,r)}f(x)-f(x)|=|P_{Q(x,r)}[f-f(x)](x)|\leq
%c\fint_{Q(x,r)}|f(y)-f(x)|dy.
%\end{equation*}
%The limit of the right hand side is zero at Lebesgue points of $f$.
We claim that, for every
$x\in 32 a(Q)\cap S$, where $Q\in\mathcal{W}_S\cap \mathcal{D}_{i+5}$,
\begin{equation}\label{lambda_claim}
\begin{split}
\lambda_Q:&=\bigg\{\fint_{32a(Q)\cap S}
\lvert f- P_{32a(Q)} f\rvert^p\,d\mathcal{H}^d\bigg\}^{1/p}\\
&\lesssim \mathcal{E}_k(f,Q_x^i)_{L^p(\R^n)}+\sum_{j=i+1}^\infty\bigg(\fint_{Q_x^{i+1}\cap S}\mathcal{E}_k(f,Q_z^j)^p_{L^p(\R^n)}\,d\mathcal{H}^d(z)\bigg)^{1/p}\,.
\end{split}
\end{equation}
In order to prove \eqref{lambda_claim}, we first estimate
\begin{equation}\label{first_second}
\lambda_Q^p\lesssim \fint_{Q_x^{i+1}\cap S} 
\lvert f-P_{Q_x^i}f\rvert^p \,d\mathcal{H}^d + \fint_{32a(Q)\cap S} \lvert P_{Q_x^i} f - 
P_{32a(Q)}f\rvert^p\,d\mathcal{H}^d\,,
\end{equation}
where we used the first part of inclusion $32a(Q)\subset Q_x^{i+1}\subset Q_x^i$. These
inclusions
and property i) allow us to bound the second term
on the right hand side by a constant multiple of
$\mathcal{E}_k(f,Q_x^i)_{L^p(\R^n)}^p$.

Consider the first term on the right hand side of \eqref{first_second},
and let $z\in Q_x^{i+1}\cap S$ be a Lebesgue point of $f$. Then, by property ii),
\[
\lvert f(z)-P_{Q_x^i}f(z)\rvert\le |P_{Q_x^i}f(z)-P_{Q_z^{i+1}}f(z)\rvert+\sum_{j=i+1}^\infty \lvert P_{Q_z^{j}}f(z)-P_{Q_z^{j+1}}f(z)\rvert\,.
\]
Since $z\in Q_z^{i+1}\subset Q_x^i$, the property i) and H\"{o}lder's inequality give
\[
\lvert P_{Q_x^i}f(z)-P_{Q_z^{i+1}}f(z)\rvert \lesssim \bigg(\fint_{Q_x^{i}}
\lvert f-P_{Q_x^i}f\rvert^p\,dy\bigg)^{1/p}\lesssim \mathcal{E}_k(f,Q_x^i)_{L^p(\R^n)}\,.
\]
Similarly
$\lvert P_{Q_z^{j}}f(z)-P_{Q_z^{j+1}}f(z)\rvert \lesssim \mathcal{E}_k(f,Q_z^j)_{L^p(\R^n)}$ for $j\in \{i+1,\ldots\}$.
We have shown that \[
\lvert f(z)- P_{Q_x^i}f(z)\rvert \lesssim \mathcal{E}_k(f,Q_x^i)_{L^p(\R^n)}+\sum_{j=i+1}^\infty\mathcal{E}_k(f,Q_z^j)_{L^p(\R^n)}\,
\]
where $z\in Q_x^{i+1}\cap S$ is a Lebesgue point of $f$.
Since $\mathcal{H}^d$-almost every point in $S$ is a Lebesgue point of $f$, we can average the last inequality over the set $Q_x^{i+1}\cap S$.  This
gives
\begin{align*}
&\bigg(\fint_{Q_x^{i+1}\cap
S} \lvert f(z)- P_{Q_x^i}f(z)\rvert^p\,d\mathcal{H}^d(z)\bigg)^{1/p}\\
&\lesssim \mathcal{E}_k(f,Q_x^i)_{L^p(\R^n)}+\sum_{j=i+1}^\infty\bigg(\fint_{Q_x^{i+1}\cap S}\mathcal{E}_k(f,Q_z^j)^p_{L^p(\R^n)}\,d\mathcal{H}^d(z)\bigg)^{1/p}.
\end{align*}
This concludes the proof of \eqref{lambda_claim}.

The family $\{32a(Q)\,:\,Q\in\mathcal{W}_S\cap\mathcal{D}_{i+5}\}$ has a bounded overlapping property. Hence, by \eqref{ii)} and \eqref{lambda_claim},
the left hand side of \eqref{EqLocalApInLpSViaLocalApInLp_est}
can be estimated by a constant multiple of
\begin{align*}
&2^{id/p} \bigg\{ \int_{S}
\sum_{Q\in\mathcal{W}_{S}\cap \mathcal{D}_{i+5}}
\lambda_Q^p\cdot \chi_{32a(Q)}(x) \,d\mathcal{H}^d(x)
\bigg\}^{1/p}\\
%&\|\mathcal{E}_k(f,Q(\cdot,2^{-i-1}))_{L^u(S)}\|_{L^p(S)}\\
&\lesssim
2^{id/p}\bigg(\int_S\mathcal{E}_k(f,Q_x^{i})_{L^p(\R^n)}^p\,d\mathcal{H}^d(x)\bigg)^{1/p}\\&\qquad\qquad+2^{id/p}\sum_{j=i+1}^\infty\bigg(\int_S\fint_{Q_x^{i+1}\cap
S}\mathcal{E}_k(f,Q_z^j)^p_{L^p(\R^n)}\,d\mathcal{H}^d(z)d\mathcal{H}^d(x)\bigg)^{1/p}.
\end{align*}
Using Fubini's theorem and Definition \ref{sset_def}, we get
\begin{align*}
&\int_S\fint_{Q_x^{i+1}\cap
S}\mathcal{E}_k(f,Q_z^j)^p_{L^p(\R^n)}\,d\mathcal{H}^d(z)d\mathcal{H}^d(x)\\&\lesssim 2^{id}\int_S \int_{S} \mathcal{E}_k(f,Q_z^j)^p_{L^p(\R^n)}\chi_{Q_x^{i+1}}(z)\,d\mathcal{H}^d(x)d\mathcal{H}^d(z)\\
&\lesssim  2^{id}\int_{S}\mathcal{H}^d(Q_z^{i+1}\cap S)\mathcal{E}_k(f,Q_z^j)^p_{L^p(\R^n)}d\mathcal{H}^d(z)\\&\lesssim \int_{S}\mathcal{E}_k(f,Q_z^j)^p_{L^p(\R^n)}\,d\mathcal{H}^d(z).
\end{align*}
Collecting the estimates above, we obtain inequality \eqref{EqLocalApInLpSViaLocalApInLp_est}.
\end{proof}

For $f\in L^p_{\rm loc}(\R^n)$ and $s>0$ denote 
\[f^\sharp_{s,p}(x)=\sup_{0<r\le 1} r^{-s}\mathcal{E}_k(f,Q(x,r))_{L^p(\R^n)},\quad x\in\R^n,\,k=[s]+1\,.\]

\begin{lemma}\label{subestTL}
Let $S$ be a $d$-set in $\R^n$, $n-1<d<n$.
Suppose that $1< p<\infty$, $s>0$, $k=[s]+1$ and $f\in L^p_{\rm loc}(\R^n)$. Then, for every $i\in \{2,3,\ldots\}$,
\begin{align*}
\|\mathcal{E}_k(f,Q(\cdot,2^{-i}))_{L^p(\R^n)}\|_{L^p(S)}
\le c2^{-i(s -(n-d)/p)}\bigg(\int_{\cup \{Q\in\mathcal{W}_{S}\,:\, \ell(Q)\overset{\kappa}\sim 2^{-i-4}\}}f^\sharp_{s,p}(x)^pdx\bigg)^{1/p},
\end{align*}
where the constant $c$ depends on $s$, $p$, $n$ and $S$.
\end{lemma}

\begin{proof}
Let $\mathcal{F}=\{Q(x,2^{-i-3})\,:\,x\in S\}$.
By the $5r$-covering theorem, see e.g. \cite[p. 23]{Mat95}, there are
disjoint cubes $Q_m=Q(x_m,2^{-i-3})\in \mathcal{F}$,
$m=1,2,\ldots$ (if there are only a finite number of cubes we change the indexing), 
such that $S\subset \bigcup_{m=1}^{\infty} 5Q_m$. 
Hence
\begin{align*}
I:&=\int_{S}\mathcal{E}_k(f,Q(x,2^{-i}))_{L^p(\R^n)}^p\,d\mathcal{H}^d(x)\\&\le \sum_{m=1}^{\infty} \int_{5Q_m\cap S} \mathcal{E}_k(f,Q(x,2^{-i}))_{L^p(\R^n)}^p\,d\mathcal{H}^d(x).
\end{align*}
Notice that, if $x\in 5Q_m=Q(x_m,(5/8)2^{-i})$, then
$Q(x,2^{-i})\subset Q(x_m,2^{-i+1})$.
%\subset Q(x,2^{-i+2})\,\] 
By \eqref{eqMonotonyOfLocalApproxSset},
\[
\mathcal{E}_k(f,Q(x,2^{-i}))_{L^p(\R^n)}\le c\mathcal{E}_k(f,Q(x_m,2^{-i+1}))_{L^p(\R^n)}\,.
%\le c\mathcal{E}_k(f,Q(x,2^{-i+2}))_{L^p(\R^n)}.
\]
Using the observation above and Definition \ref{sset_def} we can continue as follows:
\begin{equation}\label{EstimateLemsubest}
\begin{split}
I&\le c\sum_{m=1}^{\infty} \mathcal{H}^d(5Q_m\cap S) \mathcal{E}_k(f,Q(x_m,2^{-i+1}))_{L^p(\R^n)}^p\\&\le c2^{-id}\sum_{m=1}^{\infty} \mathcal{E}_k(f,Q(x_m,2^{-i+1}))_{L^p(\R^n)}^p.\\
\end{split}
\end{equation}
By 
Remark \ref{porous_obs} and
Lemma \ref{cop}, for every $Q_m$, there is  $R_m\in\mathcal{W}_{S}$ such that
$\ell(R_m)\overset{\kappa}{\sim} 2^{-i-4}$ and 
$R_m\subset Q_m$. 
If $x\in R_m$, then $Q(x_m,2^{-i+1})\subset Q(x,2^{-i+2})$ and therefore
\begin{equation}\label{pp3}
\mathcal{E}_k(f,Q(x_m,2^{-i+1}))_{L^p(\R^n)}\le c \mathcal{E}_k(f,Q(x,2^{-i+2}))_{L^p(\R^n)}\le c2^{-is}f^\sharp_{s,p}(x).
\end{equation}
Since $\ell(R_m)\overset{\kappa}{\sim} 2^{-i-4}$, we have $|R_m|\ge c2^{-in}$. By using this and \eqref{pp3}, we get
\begin{align*}
I&\le c2^{i(n-d)}\sum_{m=1}^{\infty} |R_m|\mathcal{E}_k(f,Q(x_m,2^{-i+1}))_{L^p(\R^n)}^p\\
&\le c2^{i(n-d-s p)}\sum_{m=1}^{\infty}\int_{R_m}f^\sharp_{s,p}(x)^p\,dx\\
&\le c2^{i(n-d-s p)}\int_{\cup \{Q\in\mathcal{W}_{S}\,:\, \ell(Q)\overset{\kappa}\sim 2^{-i-4}\}}f^\sharp_{s,p}(x)^pdx.
\end{align*}
Taking the $p$'th roots yields the required estimate.
\end{proof}

We are now ready for the proof of Theorem \ref{hardy_general}.

\begin{proof}
Since $F^{s}_{pq}(\R^n)\subset F^{s}_{p\infty}(\R^n)$ boundedly,  
it suffices to verify 
that
\begin{equation}\label{tama}
H:=\bigg(\int_{\R^n} \frac{\lvert f(x)\rvert^p}{\dist(x,S)^{sp}}\,dx\bigg)^{1/p}\lesssim 
\lVert f\rVert_{L^p(\R^n)} + 
\lVert f_{s,p}^\sharp \rVert_{L^p(\R^n)}\,,
\end{equation}
recall the definition of $F^{s}_{p\infty}(\R^n)$ in \S \ref{spaces}.
By properties of Whitney cubes $Q\in\mathcal{W}_{S}$, we can bound
$H^p$ by a constant multiple of
\begin{align*}
\lVert f\rVert_{L^p(\R^n)}^p + \sum _{\substack{Q\in \W_{S}\\\ell(Q)\le 2^{-7}}}
\diam (Q)^{n-sp} \fint_Q \lvert f(x)-P_{32a(Q)}f(x)\rvert^p + \lvert P_{32a(Q)}f(x)\rvert^p\,dx\,.
\end{align*}
By properties of the projection operator and monotonicity 
\eqref{eqMonotonyOfLocalApproxSset}
of local approximations,
\begin{align*}
&\sum _{\substack{Q\in \W_{S}\\\ell(Q)\le 2^{-7}}}
\diam (Q)^{n-sp} \fint_Q \lvert f(x)-P_{32a(Q)} f(x)\rvert^p\,dx\\
&\lesssim \sum _{\substack{Q\in \W_{S}\\\ell(Q)\le 2^{-7}}}
\diam (Q)^{n-sp} \mathcal{E}_k(f,32a(Q))_{L^p(\R^n)}^p\lesssim 
\lVert f^\sharp_{s,p}\rVert^p_{L^p(\R^n)}\,.
\end{align*}
Recall that $Q\subset 32a(Q)$, and their measures are comparable. 
By a Remez-type inequality,
see Theorem \ref{qremez}, and the fact that $\mathcal{H}^d$ 
almost every point $x\in S$ satisfies $\bar f(x)=0$,
\begin{align*}
&\sum _{\substack{Q\in \W_{S}\\\ell(Q)\le 2^{-7}}}
\diam (Q)^{n-sp} \fint_Q \lvert P_{32a(Q)} f(x)\rvert^p\,dx\\
&\lesssim
\sum_{i=2}^\infty   2^{i(sp-n)}\sum _{Q\in \W_{S}\cap \mathcal{D}_{i+5}}
\fint_{32a(Q)\cap S} \lvert \bar f(z) - P_{32a(Q)} \bar{f}(z)\rvert^p\,d\mathcal{H}^d(z)
=:    \Sigma\,.
\end{align*}
Recall that $\mathcal{H}^d$ almost every point is a Lebesgue point of $\bar f$.
Hence, by Lemma \ref{EqLocalApInLpSViaLocalApInLp}, Lemma \ref{subestTL}, 
%$i\in\N$,
%\begin{align*}
%&\Vert\mathcal{E}_k(f,Q(\cdot,2^{-i-1}))_{L^1(S)}\Vert_{L^p(S)}\le
%c\sum_{j=i}^\infty\Vert\mathcal{E}_k(f,Q(\cdot,2^{-j}))_{L^1(\R^n)}\Vert_{L^p(S)}\\
%&\le c\sum_{j=i}^\infty2^{-j(s -(n-d)/p)}\bigg(\int_{\cup \{Q\in\mathcal{W}_{S}\,:\, \ell(Q)\overset{\kappa}\sim 2^{-j-4}\}}f^\sharp_{s,p}(x)^pdx\bigg)^{1/p}.
%\end{align*}
and Lemma \ref{LemmaHardyInequality} with $\sigma=s p-(n-d)>0$,
\begin{equation}\label{dsn}
\begin{split}
\Sigma &\lesssim \sum_{i=2}^\infty 2^{i(s p-(n-d))}\bigg(\sum_{j=i}^\infty2^{-j(s -(n-d)/p)}\bigg(\int_{\cup \{Q\in\mathcal{W}_{S}\,:\, \ell(Q)\overset{\kappa}\sim 2^{-j-4}\}}f^\sharp_{s,p}(x)^pdx\bigg)^{1/p}\bigg)^p
\\&\lesssim  \sum_{i=2}^\infty\int_{\cup \{Q\in\mathcal{W}_{S}\,:\, \ell(Q)\overset{\kappa}\sim 2^{-i-4}\}}f^\sharp_{s,p}(x)^pdx.
\end{split}
\end{equation}

Let us
denote $U_i:=\cup \{\mathrm{int}\,Q\,:\,Q\in\mathcal{W}_{S}\text{ and }\ell(Q)\overset{\kappa}\sim 2^{-i-4}\}$, and
choose $k_0\in \N$ such that $2^{-k_0}<1/5\kappa$.
Then we claim that
\begin{equation}\label{dis}
U_i\cap U_{i'}=\emptyset,\qquad \text{ if }i\not=i'\text{ and }i,i'\equiv k\,\mathrm{mod}\,k_0.
\end{equation}
To verify this claim, let $i>i'$ be such that $i,i'\equiv k\,\mathrm{mod}\,k_0$. Then $i-i'\ge k_0$.
In particular, if $Q\in \mathcal{W}_{S}$, $\ell(Q)\overset{\kappa}{\sim} 2^{-i-4}$, and 
$Q'\in \mathcal{W}_{S}$, $\ell(Q')\overset{\kappa}{\sim} 2^{-i'-4}$, then
\[
\diam(Q)\le 2^{-i-4}\le 2^{-k_0}2^{-i'-4}<2^{-i'-4}/5\kappa\le \diam(Q').
\]
It follows that $Q\not=Q'$. Since the interiors of Whitney cubes are pairwise disjoint, we find that 
$\mathrm{int}\,Q\cap \mathrm{int}\,Q'=\emptyset$. Hence, \eqref{dis} holds.

From \eqref{dis} it follows that
\begin{equation}\label{dsm}
\begin{split}
\sum_{i=2}^\infty \int_{\cup \{Q\in\mathcal{W}_{S}\,:\, \ell(Q)\overset{\kappa}\sim 2^{-i-4}\}}f^\sharp_{s,p}(x)^pdx
&=\sum_{k=0}^{k_0-1}  \sum_{\substack{i\ge 2\\ i\equiv k\,\mathrm{mod}\,k_0}}\int_{U_i}f^\sharp_{s,p}(x)^pdx
\\&\le k_0\|f_{s,p}^\sharp\|_{p}^p.
\end{split}
\end{equation}
Combining the estimates \eqref{dsn} and \eqref{dsm}, we find that
$\Sigma\le c\|f_{s,p}^\sharp\|_p^p$.
%
%Inequality \eqref{tama} follows from the previous considerations.
\end{proof}

\section{Extension problems}\label{application}

%Suppose that $\Omega$ is a bounded domain in $\R^n$.
%For appropriate parameters, the Triebel--Lizorkin space
%in this domain is defined as follows,
%\begin{align*}
%F^s_{pq}(\Omega)=\big\{f\in L^p(\Omega)\,:\,&f=g|_\Omega\text{ for
%some } g\in F^{s}_{pq}(\R^n)\big\},\\
%\lVert f\rVert_{F^{s}_{pq}(\Omega)}&=
%\inf \lVert g\rVert_{F^{s}_{pq}(\R^n)},
%\end{align*}
%where the infimum is taken over all functions 
%$g\in F^{s}_{pq}(\R^n)$, $g|_\Omega=f$.
As an application of Hardy type inequality, we study
certain extension problems. 

\subsection{Extension by zero}
First we study the problem of zero extension.
For instance, Corollary \ref{ext_corollary} shows that
the characteristic function
$\chi_\Omega$ 
of a domain whose boundary is a $d$-set, $n-1<d<n$, is a pointwise 
multiplier in the subspace 
$\{f\in F^{s}_{pq}(\R^n)\,:\,\mathrm{Tr}_{\partial\Omega}=0\}$ if $s>(n-d)/p$.

%\begin{theorem}
%Let $f\in F^s_{pq}(\Omega)$ be a function 
%\end{theorem}

\begin{proposition}\label{ext_hardy}
Let $\Omega$ be a domain in $\R^n$ whose boundary is porous
(in particular, it suffices that $\partial\Omega$ is  a $d$-set with $d<n$).
 Let $f\in F^{s}_{pq}(\R^n)$, $1<p<\infty$, $1\le q< \infty$, $s>0$.
Then 
\begin{equation}
\lVert f\chi_\Omega \rVert_{F^{s}_{pq}(\R^n)}\lesssim \lVert f\rVert_{F^{s}_{pq}(\R^n)}+\bigg(\int_\Omega\frac{|f(y)|^p}{\dist(y,\partial\Omega)^{s p}}\,dy\bigg)^{1/p}.
\end{equation}
The implied constant depends on $p$, $q$, $s$, $n$, $d$, and $\partial\Omega$.
\end{proposition}

\begin{proof}
For convenience,  denote $\tilde f=f\chi_\Omega$.
By \cite[Theorem 2.2.2]{T89} or
\cite[Remark 3.4]{ihnatsyeva},
the norm $\|\tilde{f}\|_{F^{s}_{pq}(\R^n)}$ is equivalent to the quantity
\[
\bigg\|\bigg( \sum_{j=0}^\infty 2^{js q} \mathcal{E}_k(\tilde f,Q(\cdot,2^{-j}))_{L^1(\R^n)}^q\bigg)^{1/q} \bigg\|_{L^p(\R^n)}+\|\tilde f\|_{L^p(\R^n)}\,,\qquad k=[s]+1\,.
\]
The second summand is clearly controlled, and we focus on the first one.
By  monotonicity 
of local approximations  \eqref{eqMonotonyOfLocalApproxSset}, if $j\in\N_0$,
\[
\mathcal{E}_k(\tilde f,Q(x,2^{-j}))_{L^1(\R^n)}
\lesssim \sum_{Q\in\mathcal{D}_j} \chi_Q(x) \mathcal{E}_k(\tilde f,4Q)_{L^1(\R^n)},\quad x\in\R^n.
\]
Next we split the
summation on the right hand side in two parts,
depending on whether or not $Q\in\mathcal{C}:=\mathcal{C}_{\partial\Omega,\gamma}$
with $\gamma=5$, recall definition \eqref{c_definition}.

Observe that $4Q\cap \partial\Omega=\emptyset$ if
$Q\in \mathcal{D}_j\setminus \mathcal{C}$ with $j\ge 0$. Thus, 
for such cubes, we have
either $4Q\subset \Omega$ or $4Q\subset \R^n\setminus \overline{\Omega}$.
In both cases,
\begin{align*}
\mathcal{E}_k(\tilde f,4Q)_{L^1(\R^n)} \lesssim \mathcal{E}_k (f,4Q)_{L^1(\R^n)}\,.
\end{align*}
Hence,
\begin{equation}\label{first_ineq}
\begin{split}
&\bigg\| \bigg( \sum_{j=0}^\infty 2^{js q}
\big(\sum_{Q\in\mathcal{D}_j\setminus \mathcal{C}} \chi_Q\mathcal{E}_k(\tilde f,4Q)_{L^1(\R^n)}
\big)^q
\bigg)^{1/q} \bigg\|_{L^p(\R^n)}
\\&\lesssim
\bigg\| \bigg( \sum_{j=0}^\infty 2^{js q}
\big(\sum_{Q\in\mathcal{D}_j} \chi_Q\mathcal{E}_k(f,4Q)_{L^1(\R^n)}
\big)^q
\bigg)^{1/q} \bigg\|_{L^p(\R^n)} \lesssim \lVert f\rVert_{F^{s}_{pq}(\R^n)}\,.
\end{split}
\end{equation}
The last step follows from monotonicity of local approximations.

In order to estimate the remaining
term, associated with cubes $Q\in\mathcal{C}$, we use 
Theorem \ref{reverse}. Note also that cubes in $\mathcal{D}_j$ have mutually disjoint interiors. Thus, we have
\begin{equation}\label{second_ineq}
\begin{split}
A:&=\bigg\| \bigg( \sum_{j=0}^\infty 2^{js q}
\big(\sum_{Q\in\mathcal{D}_j\cap \mathcal{C}} \chi_Q\mathcal{E}_k(\tilde f,4Q)_{L^1(\R^n)}
\big)^q
\bigg)^{1/q} \bigg\|_{L^p(\R^n)}\\
&\lesssim  \bigg(\int_{\R^n} \sum_{Q\in\mathcal{C}}\chi_Q(x) \ell(Q)^{-s p} \mathcal{E}_k(\tilde f,4Q)_{L^1(\R^n)}^p\,dx\bigg)^{1/p}\,.
\end{split}
\end{equation}
By the following inequality,
\[
\mathcal{E}_k(\tilde f,4Q)_{L^1(\R^n)}^p \lesssim \fint_{4Q} \lvert \tilde f\rvert^p\,dx= 
\frac{1}{\lvert 4Q\rvert} \int_{4Q\cap \Omega} \lvert f\rvert^p\,dx\,,
\]
and definition of family $\mathcal{C}$,
we obtain
\begin{align*}
A^p\lesssim \int_{\Omega} \bigg\{\sum_{Q\in\mathcal{C}} \ell(Q)^{-s p}\chi_{4Q}(y) \bigg\}  \lvert
f(y)\rvert^p\,dy\lesssim \int_\Omega\frac{|f(y)|^p}{\dist(y,\partial\Omega)^{s p}}\,dy\,.
\end{align*}
This completes the proof.
\end{proof}

The following 
is a consequence of Theorem \ref{hardy_general} and Proposition \ref{ext_hardy}.
See also Remark \ref{porous_obs}

\begin{corollary}\label{ext_corollary}
Suppose that $\Omega$ is a domain in $\R^n$ whose
boundary is a $d$-set, $n-1<d<n$.
Suppose also that $1<p<\infty$, $1\le q< \infty$, and
$s>(n-d)/p$. 
Let $f\in F^{s}_{pq}(\R^n)$
be such
that $\mathrm{Tr}_{\partial\Omega} f = 0$ pointwise
$\mathcal{H}^d$ almost everywhere.
Then 
\[
\lVert f\chi_\Omega \rVert_{F^{s}_{pq}(\R^n)}\lesssim \lVert f\rVert_{F^{s}_{pq}(\R^n)}\,,
\]
where the implied constant depends on $p$, $q$, $s$, $n$, $d$, and $\partial\Omega$.
\end{corollary}

An application of the corollary is Theorem \ref{intrinsic}.
In order
to formulate this theorem, we  recall some notation which is common in the literature on 
function spaces on domains, \cite{T2,T3}. 

Let $\Omega$ be an open set in $\R^n$, $1\le p<\infty$,  $1\le q\leq\infty$,
and $s>0$. Then
\[
F^s_{pq}(\Omega)=\{f\in L^p(\Omega)\,:\, \text{there is a}\, g\in F^s_{pq}(\R^n)\,\,\text{with}\,g|_\Omega=f\}
\]
\[
\Vert f\Vert_{F^s_{pq}(\Omega)}=\inf\Vert g\Vert_{F^s_{pq}(\R^n)},
\]
where the infimum is taken over all $g\in F^s_{pq}(\R^n)$ such that $g|_\Omega=f$ pointwise a.e. 
As usual, we also denote
\begin{equation}\label{definitionFtilde}
\widetilde{F}^s_{pq}(\Omega)=\{f\in L^p(\Omega):\, \text{there is a}\, g\in F^s_{pq}(\R^n)\,\,\text{with}\,g|_\Omega=f \, \text{and} \, \supp g\subset\overline{\Omega}\}
\end{equation}
\[
\Vert f\Vert_{\widetilde{F}^s_{pq}(\Omega)}=\inf\Vert g\Vert_{F^s_{pq}(\R^n)},
\]
where the infimum is taken over all $g$ admitted in \eqref{definitionFtilde},

Finally, 
$\overset\circ{ F^{s}_{pq}}(\Omega)$ is a completion of $C^\infty_0(\Omega)$ in $F^s_{pq}(\Omega)$.

\begin{theorem}\label{intrinsic}
Let $\Omega$ be a domain 
whose closure $\overline{\Omega}$ is an $n$-set, and whose
boundary $\partial\Omega$ is a $d$-set with $n-1<d<n$. Suppose that
$1<p<\infty$, $1\le q<\infty$ and $s>(n-d)/p$.
Then 
\[
\overset\circ{ F^{s}_{pq}}(\Omega)\subset\widetilde{F}^s_{pq}(\Omega)\,,
\]
and this inclusion is bounded. 
\end{theorem}

\begin{remark}\label{clarify}
Let us first clarify the role of assumptions in the theorem.
Suppose that $f\in F^{s}_{pq}(\Omega)$ and
$g\in F^{s}_{pq}(\R^n)$ is any extension of $f$.
Recall that $\mathcal{H}^d$ almost
every point in $\partial\Omega$ is a Lebesgue point of
$g$.
By Lebesgue differentiation theorem,
and the assumption that $\overline{\Omega}$ is an $n$-set,
\[
\mathrm{Tr}_{\partial\Omega} g(x)=\lim_{r\to 0+}\fint_{Q(x,r)\cap \Omega} f(y)\,dy
\]
in the Lebesgue points $x\in\partial\Omega$ of $g$. 
Here we also used the fact that $\partial\Omega$ has zero $n$-measure.
To state the conclusion otherwise,
any extension of $f$ has the same trace $\mathcal{H}^d$ a.e. on $\partial\Omega$, and
this trace coincides a.e. with the {\em interior trace} given above.
\end{remark}

\begin{proof}[Proof of Theorem \ref{intrinsic}]
Suppose that $f\in \overset\circ{ F^{s}_{pq}}(\Omega)$, then there is a sequence $f_j\in C^\infty_0(\Omega)$ such that 
$f_j\to f$ in $F^s_{pq}(\Omega)$. By definition, there
are  functions $g$ and
$\{G_j\}_{j\in\N}$ belonging to $F^s_{pq}(\R^n)$ for which
$g\lvert_\Omega = f$
and $G_j\lvert_\Omega = f-f_j$, $j\in\N$. Moreover,
we can suppose that $\lVert g\rVert_{F^s_{pq}(\R^n)}\le 2\lVert f\rVert_{F^s_{pq}(\Omega)}$ and $\lVert G_j\rVert_{F^s_{pq}(\R^n)}\le 2\lVert f-f_j\rVert_{F^s_{pq}(\Omega)}$ for all $j$.

Since $(g-G_j)\lvert_\Omega = f_j$, 
and the trace on the boundary is independent of the extension,
the trace of $g-G_j$ on $\partial\Omega$ vanishes. Thus,
\begin{align*}
\lVert \mathrm{Tr}_{\partial\Omega} g\rVert_{L^p(\partial\Omega)}
&= \lim_{j\to\infty}\lVert \mathrm{Tr}_{\partial\Omega} g-\mathrm{Tr}_{\partial\Omega} (g-G_j)\rVert_{L^p(\partial\Omega)}.
\end{align*}
By linearity and boundedness of the trace operator, \cite[Theorem 4.8]{ihnatsyeva},
\begin{align*}
\lVert \mathrm{Tr}_{\partial\Omega} g\rVert_{L^p(\partial\Omega)}&=\lim_{j\to\infty} \lVert \mathrm{Tr}_{\partial\Omega} G_j\lVert_{L^p(\partial\Omega)} \\
&\lesssim \lim_{j\to\infty}\lVert G_j\rVert_{F^s_{pq}(\R^n)} = 0.
\end{align*}
We have shown that $g$ has zero trace on $\partial\Omega$.
By Corollary \ref{ext_corollary},
\[
\lVert g\chi_\Omega\lVert_{F^s_{pq}(\R^n)} \lesssim \lVert g\rVert_{F^s_{pq}(\R^n)} \le 2 \lVert f\rVert_{F^s_{pq}(\Omega)}\,,
\]
which is a sufficient estimate since $(g\chi_\Omega)\lvert_\Omega= f$, and
the support of $g\chi_\Omega$ is contained in $\overline{\Omega}$.
\end{proof}

\begin{remark}
Theorem \ref{intrinsic}
is related to the following result due to Caetano, 
\cite[Corollary 2.7]{caetano}. Let $\Omega$ be a bounded domain
such that $\partial\Omega$ is a $d$-set for some $d<n$.
Then \begin{equation}\label{test}
\overset\circ{ F^{s}_{pq}}(\Omega) = F^{s}_{pq}(\Omega),\qquad s<(n-d)/p\,.
\end{equation}
This identification fails if $p>1$, $s>(n-d)/p$, and
${\Omega}$ is an $n$-set
whose boundary is a $d$-set with $d<n$,
\cite[Proposition 3.7]{caetano}.
Theorem \ref{intrinsic} gives a partial counterpart 
of identification \eqref{test} in case of $s>(n-d)/p$.
\end{remark}

\subsection{Extension from the boundary trace}
We close this paper by considering an extension of a smooth function $f$
from its trace $\mathrm{Tr}_{\partial\Omega} f$ on the boundary of a given domain $\Omega$.
In the complement of the domain, an extension is defined
in terms of polynomial projections 
of  $\mathrm{Tr}_{\partial\Omega} f$.
Throughout, we assume that the boundary $\partial\Omega$
is a $d$-set with $n-1<d<n$.

To each cube $Q=Q(x_Q,r_Q)\in\mathcal{W}_{\partial\Omega}$ in the Whitney
decomposition of $\R^n\setminus \partial\Omega$,
we assign a nearby cube
$a(Q):=Q(a_Q,r_Q/2)$,
where $a_Q\in \partial\Omega$ is such that
$\|x_Q-a_Q\|_\infty =\mathrm{dist}(x_Q,\partial\Omega)$. 
%Let $\Delta\ge 1$ be a parameter which will be chosen later. Then 
%\begin{equation}\label{ii)}
%\mathcal{H}^d(a(Q)\cap \partial\Omega)\geq 
%c\, r_Q^d,\qquad \text{if }\diam(Q)\le\Delta\,.
%\end{equation}
%This follows from \eqref{sset_rem}, 
%and the constant $c>0$ depends on  $\Delta$, $S$, and $n$.
Let $\{P_\beta\}_{|\beta|\le k}$  be an
orthonormal basis of $\mathcal{P}_{k}$, $k\geq 0$, with
respect to the inner product
\[\langle p,q\rangle= \int_{a(Q)\cap \partial\Omega} pq\, d\mathcal{H}^d,\qquad p,q\in \mathcal{P}_k.\]
Observe that the zero set of $p\in\mathcal{P}_{k}\setminus \{0\}$
has Hausdorff dimension at most $n-1$. Hence, the formula gives
 an inner product.
Define a linear operator $\mathrm{Pr}_{k,a(Q)}:L^1(a(Q)\cap \partial\Omega)\to \mathcal{P}_{k}$ by setting
\begin{equation}\label{projs}
\mathrm{Pr}_{k,a(Q)} f :=  \sum_{|\beta|\le k} \langle f,P_\beta\rangle P_\beta
=\sum_{|\beta|\le k} \bigg(\int_{a(Q)\cap \partial\Omega} fP_\beta\,d\mathcal{H}^d\bigg) P_\beta
\end{equation}
if $\diam(Q)\le \Delta:=16000$,
and $\mathrm{Pr}_{k,a(Q)} f=0$ 
otherwise.
%This is a projection if $\diam(Q)\le \Delta$.

Let $\{\varphi_Q:Q\in \mathcal{W}_{\partial\Omega}\}$ be a smooth partition of unity,
subordinate to the Whitney decomposition $\mathcal{W}_{\partial\Omega}$.
Then, in particular, 
\[
\chi_{\R^n\setminus {\partial\Omega}}= \sum_{Q\in\mathcal{W}_{\partial\Omega}} \varphi_Q\,\quad \text{and }\quad \mathrm{supp}\,\varphi_Q\subset (9/8)Q\,, \text{ if } Q\in\mathcal{W}_{\partial\Omega}.
\]
%and 
%$\mathrm{supp}\,\varphi_Q\subset \frac{9}{8}Q$ for every
%$Q\in\mathcal{W}_{\partial\Omega}$.
%For every $Q\in\mathcal{W}_{\partial\Omega}$ and $k\in\N_0$, define a linear operator $P_{k,Q}:L^1(a(Q)\cap {\partial\Omega})\to \mathcal{P}_{k}$ by
%\[
%P_{k,Q}:=
%\begin{cases}
%P_{k,a(Q)},\quad &\text{if } \diam(Q)\le \Delta;\\
%0,\quad &\text{if }{\rm diam}(Q)>\Delta.
%\end{cases}
%\] 
%Here (say) $\Delta:=16000$ and
%the $P_{k,a(Q)}:L^1(a(Q)\cap {\partial\Omega})\to \mathcal{P}_{k}$
%is a polynomial projection operator, as defined above.
For $f\in F^{s}_{pq}(\R^n)$ with $s>(n-d)/p$
and $k=[s]+1$, we define \begin{equation}\label{DefExtensionOperator_I}
{\rm Ext}_{k,\Omega}f(x):=\begin{cases}
f(x), \quad &\text{if }x\in \Omega;\\
\sum_{Q\in \mathcal{W}_{\partial\Omega}}\varphi_Q(x)\{\mathrm{Pr}_{k-1,a(Q)}\circ \mathrm{Tr}_{{\partial\Omega}}(f)\}(x), \quad &\text{if }x\in\mathbb{R}^n\setminus \Omega.
\end{cases}
\end{equation}
Observe that  \eqref{DefExtensionOperator_I}
induces a linear operator ${\rm Ext}_{k,\Omega}$. 
We emphasise
that the values of $\mathrm{Ext}_{k,\Omega}f$ outside of
$\Omega$ 
depend only on the trace of $f$ on the boundary---loosely speaking, we
are extending from the boundary trace.

%Also, from the properties of dyadic cubes it follows that
%\begin{equation}\label{supp}
%\mathrm{supp}({\rm Ext}_{k,\Omega}f)\subset \bigcup_{x\in {\partial\Omega}}Q(x,8\Delta).
%\end{equation}

\begin{theorem}\label{main_extension}
Suppose that $\Omega$ is a domain whose
boundary is a $d$-set, $n-1<d<n$. Suppose that 
$1<p<\infty$, $1\le q< \infty$, $s>(n-d)/p$, and $k=[s]+1$. Then
\[
\mathrm{Ext_{k,\Omega}} \in \mathscr{L}(F^s_{pq}(\R^n))\,.
\]
That is, the extension operator is a bounded linear operator
on $F^s_{pq}(\R^n)$, and the operator
norm depends on $p$, $q$, $s$, $n$, $d$, and $\partial\Omega$.
\end{theorem}

\begin{proof}
Define
\begin{equation}\label{DefExtensionOperator}
{\rm Ext}_{k,\partial\Omega}f:=\sum_{Q\in \mathcal{W}_{\partial\Omega}}\varphi_Q\{\mathrm{Pr}_{k-1,a(Q)}\circ \mathrm{Tr}_{{\partial\Omega}}(f)\}.
\end{equation}
By  restriction and extension theorems for $d$-sets, 
\cite[Theorem 4.8]{ihnatsyeva} and  \cite[Theorem 6.7]{ihnatsyeva}, we see that
$\mathrm{Ext}_{k,\partial\Omega}$ is a bounded linear
operator on $F^s_{pq}(\R^n)$. Moreover, the function
$g:= f - \mathrm{Ext}_{k,\partial\Omega} f$
is such that $\mathrm{Tr}_{\partial\Omega} g=0$ pointwise
$\mathcal{H}^d$ almost everywhere in $\partial\Omega$, \cite[Proposition 5.5]{ihnatsyeva}. Therefore,
by Corollary \ref{ext_corollary} and the boundedness of $\mathrm{Ext}_{k,\partial\Omega}$,
\[
\lVert{\rm Ext}_{k,\Omega}f\rVert_{F^{s}_{pq}(\R^n)}
 = \lVert g\chi_\Omega + \mathrm{Ext}_{k,\partial\Omega} f\rVert_{F^s_{pq}(\R^n)}\lesssim \lVert f\rVert_{F^{s}_{pq}(\R^n)}\,.
\]
This is the desired norm estimate.
\end{proof}

\begin{remark}
The operator $\mathrm{Ext}_{1,\Omega}$ is
in considered \cite{H} for studying the extension problem
on spaces $W^{1,p}(\Omega)$. 
%Let us also recall that the
%an extension operator of Calder\'on, 
%\cite{calderon},
%extends
%any $f\in W^{k,p}_0(\Omega)$ to be zero outside of a given
%Lipschitz domain $\Omega$.
%
%We are unaware of
%references treating the higher order extension operators $\mathrm{Ext}_{k,\Omega}$ with $k> 1$.
%In particular, Theorem \ref{main_extension} is new to our knowledge.
\end{remark}

The following is a corollary of Theorem \ref{main_extension}
and Remark \ref{clarify}. 

\begin{corollary}\label{cor_ext}
Let $\Omega$ be a domain 
whose closure $\overline{\Omega}$ is an $n$-set, and whose
boundary $\partial\Omega$ is a $d$-set with $n-1<d<n$. Suppose that
$1<p<\infty$, $1\le q<\infty$, $s>(n-d)/p$, and $k=[s]+1$.
Fix any bounded extension operator $\mathrm{N}^s_{pq}:F^s_{pq}(\Omega)\to F^{s}_{pq}(\R^n)$
(possibly non-linear).
Then, the formula
\[
\mathrm{E}^s_{pq}: = f\mapsto \mathrm{Ext}_{k,\Omega} (\mathrm{N}^s_{pq} f)
\]
defines a bounded linear extension operator $F^{s}_{pq}(\Omega)\to F^{s}_{pq}(\R^n)$, which
is independent of the chosen extension operator $\mathrm{N}^s_{pq}$.
\end{corollary}

\begin{remark}
Let us clarify that, by definitions, there
is always a bounded non-linear extension operator
$\mathrm{N}^s_{pq}:F^{s}_{pq}(\Omega)\to F^s_{pq}(\R^n)$
for which
$\mathrm{N}^s_{pq} f\lvert_\Omega = f$ for every $f\in F^s_{pq}(\Omega)$.
\end{remark}

\begin{remark}
We observe the following independence
on the microscopic parameter:
the induced extension
operator $\mathrm{E}^s_{pq}$, see
Corollary \ref{cor_ext},
has the property that, for every $1\le r\le \infty$,
\[
\lVert \mathrm{E}^s_{pq}f\lvert_{\R^n\setminus \overline{\Omega}} \rVert_{F^{s}_{pr}(R^n\setminus\overline{\Omega})}
\lesssim \lVert f\rVert_{F^{s}_{pq}(\Omega)}\,.
\]
This follows from restriction and extension theorems,
\cite[Theorem 4.8]{ihnatsyeva} and  \cite[Theorem 6.7]{ihnatsyeva}. 
In particular,
we use the fact that the trace space $B^{s-(n-d)/p}_{pp}(\partial\Omega)$ 
of $F^{s}_{pq}(\R^n)$ is independent of the microscopic parameter
$q$.
\end{remark}

\bigskip

\small{
%\noindent Addresses:

\noindent L.I.: Department of Mathematics and Statistics,
P.O. Box 68, FI-00014 University of Helsinki, Finland. 
e-mail: {\tt lizaveta.ihnatsyeva@aalto.fi}

\medskip
\noindent A.V.V.: Department of Mathematics and Statistics,
P.O. Box 68, FI-00014 University of Helsinki, Finland. 
e-mail: {\tt antti.vahakangas@helsinki.fi}}


\begin{thebibliography}{BBM2}
 \bibitem{AH} D.R. Adams and L.I. Hedberg, Function spaces and potential theory, 
Grundlehren Math. Wiss.  314,
Springer-Verlag, Berlin, 1996.

\bibitem{BesovIlinNikolski} O.V. Besov, V.P. Il'in and S.M. Nikol'skii, Integral representations of functions and embedding theorems, 
New York: J. Wiley and Sons, vol. 1, 1978; vol. 2, 1979.

%\bibitem{Bojarski} B. Bojarski, 
%\textit{Remarks on Sobolev imbedding inequalities}, in: 
%Lecture Notes in Math., vol. 1351, Springer,
%New York, 1988,  pp. 52--68.

\bibitem{BMMM}
Kevin Brewster, Dorina Mitrea, Irina Mitrea, and Marius Mitrea,
\textit{Extending Sobolev Functions with Partially Vanishing Traces
from Locally $(\epsilon,\delta)$-Domains and Applications to Mixed
Boundary Problems},
arXiv: 1208.4177 (2012).

\bibitem{Brudnyi74} Yu. Brudnyi, Spaces defined by means of local approximations, Trudy Moskov. Math. Obshch.  24 (1971) 69--132; English transl.: Trans. Moskow Math Soc.  24 (1974) 73--139.

\bibitem{BrudnyiBrudnyi} 
A. Brudnyi and Yu. Brudnyi, 
\textit{Remez type inequalities and Morrey--Campanato spaces on Ahlfors regular sets},
Contemp. Math., \textbf{445} (2007), 19--44.

\bibitem{caetano}
{Ant\'onio M. Caetano},
\textit{Approximation by functions of compact support in Besov--Triebel--Lizorkin spaces on irregular domains},
{Studia Math.}, \textbf{142} (2000), no. 1,  {47--63}.

\bibitem{calderon}
{A.-P. Calder\'on},
\textit{Lebesgue spaces of differentiable functions and distributions},
In Proc. Sympos. Pure Math., Vol. IV, pages 33--49, American Mathematical
Society, Providence, R.I., 1961.

\bibitem{ds}
Ronald A. DeVore and Robert C. Sharpley,
\textit{{B}esov spaces on domains in $\mathbb{R}^d$},
{Trans. Amer. Math. Soc.}, \textbf{335}  (1993),
{843--864}.

\bibitem{DeVoreSharpley}
R.A. DeVore and R.C. Sharpley, Maximal functions measuring smoothness,
Mem. Amer. Math. Soc. \textbf{47} (1984), 1--115.

\bibitem{Dyda}
B. Dyda, \textit{ A fractional order Hardy inequality}, Illinois J. Math. \textbf{48} (2004), no. 2, 575--588.
%\bibitem[EH-S]{EdmundsHurriSyrjanen2011}
%David E. Edmunds and Ritva Hurri-Syrj\"anen,
%\textit{The improved Hardy inequality},
%{Houston J. Math.}, \textbf{37}  (2011),
%{929--937}.

\bibitem{E-HSV}
{David E. Edmunds}, {Ritva Hurri-Syrj\"anen} and {Antti V. V{\"a}h{\"a}kangas},
\textit{Fractional Hardy-type inequalities in domains with uniformly fat complement},
{Proc. Amer. Math. Soc.}, accepted for publication.

\bibitem{FaracoRogers}
Daniel Faraco and Keith M. Rogers,
\textit{The Sobolev norm of characteristic functions with applications to the Calder\'on inverse problem},
Quart. J. Math., doi: 10.1093/qmath/har039 (2012).

\bibitem{H}
{Petteri Harjulehto},
\textit{Traces and Sobolev extension domains},
{Proc. Amer. Math. Soc.}, \textbf{134}  (2006), no. 8, {2373--2382}.


\bibitem{H-SV}
{Ritva Hurri-Syrj\"anen} and {Antti V. V{\"a}h{\"a}kangas},
\textit{On fractional Poincar\'e inequalities}, J. Anal. Math., accepted for publication.

\bibitem{ihnatsyeva}
Lizaveta Ihnatsyeva and Antti V. V{\"a}h{\"a}kangas,
\textit{Characterization of traces of smooth functions on Ahlfors regular sets}, arXiv:1109.2248 (2011).

\bibitem{JonssonWallin1984} 
Alf Jonsson and Hans Wallin,
Function spaces on subsets of
$\mathbb{R}^n$,  Mathematical reports, Vol. 2,  Part 1, London: Harwood Academic Publishers,
1984.

%\bibitem{iwaniec} T. Iwaniec and C.A. Nolder, 
%\textit{Hardy--Littlewood inequality for quasiregular mappings in certain domains 
%in {$\R^n$}},
%Ann. Acad. Sci. Fenn. Ser. A I Math., \textbf{10} (1985) 267--282.

\bibitem{JonssonII} A. Jonsson, 
\textit{Atomic decomposition of Besov spaces
on closed sets}, in:  Function spaces, differential operators
and non-linear analysis, Teubner-Texte Math., vol. 133, Leipzig: Teubner, 1993,
pp. 285--289.

\bibitem{KinnunenMartio} J. Kinnunen and O. Martio, \textit{Hardy's inequalities for Sobolev functions}, Math. Res. Lett. \textbf{4} (1997), no. 4, 489--500.
\bibitem{lehrback}
Juha Lehrb\"ack,
\textit{Weighted Hardy inequalities and the size of the boundary},
Manuscripta Math., \textbf{127} (2008), no. 2, 249--273.

%\bibitem{lehrbackII}
%Juha Lehrb\"ack and Heli Tuominen,
%\textit{A note on the dimensions of Assouad and Aikawa},
% J. Math. Soc. Japan (to appear).

\bibitem{Leindler} L. Leindler,
\textit{Generalization of inequalities of Hardy and Littlewood}, 
Acta Sci. Math. (Szeged), \textbf{ 31} (1970), 279--285. 

\bibitem{Lewis1988}
John L. Lewis, \textit{Uniformly fat sets}, 
Trans. Amer. Math. Soc., \textbf{308} (1988), 177--196.

%\bibitem{Luukkainen} J. Luukkainen, 
%\textit{Assouad dimension: antifractal metrization, porous sets}, 
%and homogeneous measures, J. Korean Math. Soc., \textbf{35} (1998),  23--76.


\bibitem{Mat95}
P.\ Mattila, Geometry of sets and measures in Euclidean spaces: 
fractals and rectifiability, Cambridge Studies in Advanced Mathematics ~44, Cambridge University Press, Cambridge, 1995.

\bibitem{Shvartsman}
Pavel Shvartsman, \textit{Local approximations and intrinsic characterization
of spaces of smooth functions on regular subsets of
$\mathbb{R}^n$}, Math. Nachr.,  \textbf{279}  (2006),
1212--1241.

\bibitem{sickel}
Winfried Sickel,
\textit{On pointwise multipliers
for $F^{s}_{pq}(\R^n)$ in case $\sigma_{p,q}<s<n/p$},
Annali Mat. pura applicata, \textbf{176} (1999),  209--250.


%\bibitem[Sl]{Sloane2011}
%Craig A. Sloane, \textit{A fractional Hardy-Sobolev-Maz'ya inequality on the upper halfspace}, 
%Proc. Amer. Math. Soc., \textbf{139} (2011), 4003--4016. 


%\bibitem[MS1]{MS1}
%Vladimir Maz'ya and Tatyana Shaposhnikova, \textit{On the Bourgain, Brezis, and Mironescu theorem concerning
%limiting embeddings of fractional Sobolev spaces}, J. Funct. Anal., 
%\textbf{195} (2002), 230--238.



%\bibitem[MS2]{MS2}
%Vladimir Maz'ya and Tatyana Shaposhnikova, \textit{Erratum to ``On the Bourgain, Brezis, and Mironescu theorem concerning
%limiting embeddings of fractional Sobolev spaces''}, J. Funct. Anal., 
%\textbf{201} (2003), 298--300.

\bibitem{S}
Elias M. Stein,
\textit{Singular integrals and differentiability properties of functions},
Princeton Univ. Press, Princeton, New Jersey, 1970.

\bibitem{T89} Hans Triebel, \textit{Local approximation spaces},
Z. Anal. Anwendungen, \textbf{8} (1989), 261--288.

\bibitem{T99}
Hans Triebel,
\textit{Hardy inequalities in function spaces},
Math. Bohem., \textbf{124} (1999), 123--130.

\bibitem{T03}
Hans Triebel,
\textit{Non-smooth atoms and pointwise multipliers in function spaces},
Ann. Mat. Pura Appl., \textbf{182} (2003), 457--486.

\bibitem{T2}
Hans Triebel, Theory of Function Spaces II, Basel, Birkh\"auser, 1992.

\bibitem{T3} Hans Triebel, The Structure of Functions, Basel, Birkh\"auser, 2001.

\bibitem{T4} Hans Triebel, Function Spaces and Wavelets on Domains, European Mathematical Society, 2008.


%\bibitem{Wallin} Hans Wallin, \textit{The trace to the boundary of Sobolev spaces on a snowflake}, Manuscripta Math., \textbf{73} (1991), 117--125.

\bibitem{Wannebo1990}
Andreas Wannebo, \textit{Hardy inequalities}, 
Proc. Amer. Math. Soc., \textbf{109} (1990), 85--95.

\end{thebibliography}
\end{document}